\documentclass[a4paper,10pt,twoside]{amsart}
\usepackage[english]{babel}
\usepackage[utf8]{inputenc}
\usepackage[a4paper,inner=2.2cm,outer=2.2cm,top=4cm,bottom=4cm,pdftex]{geometry}
\usepackage{fancyhdr}
\usepackage{color}
\usepackage{bold-extra}
\usepackage{mathrsfs}
\usepackage{array}

\usepackage{comment}
\usepackage{graphics}
\usepackage{aliascnt}
\usepackage[pdftex,citecolor=green,linkcolor=red]{hyperref}
\usepackage{mathtools}

\usepackage{amsmath}
\usepackage{amsfonts}
\usepackage{amssymb}
\usepackage{amsthm}
\usepackage{enumerate}

\newtheorem{theorem}{Theorem}[section]
\newtheorem{corollary}[theorem]{Corollary}

\newtheorem{lemma}[theorem]{Lemma}
\newtheorem{prop}[theorem]{Proposition}
\theoremstyle{definition}

\newtheorem{rem}[theorem]{Remark}

\numberwithin{equation}{section}

\newcommand\R{\mathbb{R}}
\newcommand\Z{\mathbb{Z}}
\newcommand\N{\mathbb{N}}

\newcommand\eps{\varepsilon}
\renewcommand{\Re}{\textnormal{Re}}
\renewcommand{\Im}{\textnormal{Im}}

\newcommand{\sumflat}{\mathop{{\sum}^{\raisebox{-3pt}{\makebox[0pt][l]{$\flat$}}}}}

\usepackage{tikz}
\usetikzlibrary{calc}
\newlength{\lyxlabelwidth}      
	{\settowidth{\lyxlabelwidth}{#2}
		\begin{description}[font=\normalfont,style=sameline,
			leftmargin=\lyxlabelwidth,#1]}
	{\end{description}}

\begin{document}

\title{On the Real Zeroes of Half-integral Weight Hecke Cusp Forms, II}

\author{Jesse J\"a\"asaari}
\address{Department of Mathematics and Statistics\\
  University of Turku\\
  20014 Turku, Finland}
  \email{jesse.jaasaari@utu.fi}
  
\subjclass[2010]{Primary 11F37; Secondary 11M06}
  
\begin{abstract}   
We show that for $\gg K^2$ of the half-integral weight Hecke cusp forms in Kohnen plus subspaces with weight bounded by a large parameter $K$, the number of "real" zeroes grows at the expected rate. A key technical step in the proof is to obtain sharp bounds for the mollified first and second moments of quadratic twists of modular $L$-functions. 
\end{abstract}
  
\maketitle

\section{Introduction and the main result}

\noindent One of the striking consequences of the holomorphic Quantum Unique Ergodicity (QUE) conjecture of Rudnick and Sarnak \cite{Rudnick-Sarnak1994, Luo-Sarnak2004} (which is now a theorem thanks to the work of Holowinsky and Soundararajan \cite{Holowinsky-Soundararajan2010}) is that the zeroes of holomorphic Hecke cusp forms equidistribute inside the fundamental domain $\mathcal D:=\mathrm{SL}_2(\mathbb Z)\backslash\mathbb H$, where $\mathbb H$ is the upper half of the complex plane, as the weight of the form tends to infinity. This was proved by Rudnick \cite{Rudnick2005} building upon the work of Shiffman and Zelditch \cite{Shiffman-Zelditch1999} in the case of compact manifolds. Subsequent investigations concerning the distribution of zeroes in small scales were initiated by Ghosh and Sarnak \cite{Ghosh-Sarnak2012}, and then continued by Matom\"aki \cite{Matomaki2016} and Lester--Matom\"aki--Radziwi{\l\l} \cite{Lester-Matomaki-Radziwill2018}. In a recent work \cite{Jaasaari2026} the author obtained results in the spirit of the latter work in the setting of half-integral weight Hecke cusp forms. We refer to the introduction of \cite{Jaasaari2026} for the differences between the half-integral weight case and the integral weight case, but for now we only briefly mention that an important reason why the methods of \cite{Ghosh-Sarnak2012, Matomaki2016, Lester-Matomaki-Radziwill2018} do not adapt to the half-integral weight setting in a straightforward way is that the Fourier coefficients of half-integral weight Hecke cusp forms are not multiplicative. 

In this paper we are interested in the real\footnote{This notion comes from the fact that a half-integral weight form is real-valued on the lines $\delta_1$ and $\delta_2$ when it is normalised to have real Fourier coefficients.} zeroes of half-integral weight Hecke cusp forms, that is zeroes on the two geodesic segments
\[ 
\delta_1:=\left\{s\in\mathbb C:\,\Re(s)=-\frac12\right\} \qquad \text{and} \qquad \delta_2:=\left\{s\in\mathbb C:\,\Re(s)=0\right\}.
\]
\noindent For motivation and related results for integral weight Hecke cusp forms, we refer to the introduction of \cite{Jaasaari2026} as well as the papers \cite{Ghosh-Sarnak2012, Matomaki2016, Lester-Matomaki-Radziwill2018}.

Let us now summarise the main results of \cite{Jaasaari2026}. Before that we need to introduce some notation. Throughout the article, let $k$ be a positive integer. We write $S_{k+\frac12}(4)$ for the space of cusp forms of weight $k+\frac12$ and level $4$. Also we denote by $S_{k+\frac12}^+(4)\subset S_{k+\frac12}(4)$ the Kohnen plus subspace and let $B^+_{k+\frac12}$ denote a fixed Hecke eigenbasis for $S_{k+\frac12}^+(4)$. Note that for half-integral weight cusp forms we cannot normalise the first Fourier coefficient $c_g(1)$ to be equal to one without losing the algebraicity of the Fourier coefficients. This means that, unlike in the integral weight case, there is no canonical choice for $B_{k+\frac12}^+$, but this causes no problems for us. We write $\mathcal Z(g):=\{z\in\mathcal F:\,g(z)=0\}$ for the set of zeroes of $g\in S_{k+\frac12}^+(4)$ and let $\mathcal F_Y:=\{z\in\mathcal F\,:\,\Im(z)\geq Y\}$ be a Siegel set, where $\mathcal F:=\Gamma_0(4)\backslash\mathbb H$. Finally, set\footnote{Here, and throughout the paper, the notation $\ell\sim L$ means that $L\leq \ell\leq 2L$.}
\[\mathcal S_K:=\bigcup_{k\sim K}B_{k+\frac12}^+\]
and note that the cardinality of this set is $\sim K^2/4$.

With these notations\footnote{We also write $\#\mathcal X$ for the cardinality of a set $\mathcal X$.} the main results of \cite{Jaasaari2026} may be stated as follows. 

\begin{theorem}(\cite[Theorems 1.1. and 1.3]{Jaasaari2026})\label{prior-theorem} 
Let $K$ be a large parameter, $\eps>0$ be an arbitrarily small fixed number, and $j\in\{1,2\}$.
\begin{enumerate}
\item For $\gg_\eps K^2/(\log K)^{3/2+\eps}$ of the forms $g\in\mathcal S_K$ we have 
\[\#\{z\in\mathcal Z(g)\cap\delta_j\cap\mathcal F_Y\}\gg\frac K{Y}\left(\log K\right)^{-23/2}  \]
for $\sqrt{K\log K}\leq Y\leq K^{1-\delta}$ with any small fixed constant $\delta>0$. 
\item For at least $(1/2-\eps)\#\mathcal S_K$ of the forms $g\in\mathcal S_K$ we have 
\[\#\{z\in\mathcal Z(g)\cap\delta_j\cap\mathcal F_Y\}\gg\sqrt{\frac K{Y}} \]
for $\sqrt{K\log K}\leq Y\leq K^{1-\delta}$ with any small fixed constant $\delta>0$. 
\end{enumerate}
\end{theorem} 
\noindent The first part is proved by studying sign changes of Fourier coefficients $c_g(|d|)$ of a half-integral weight cusp form $g$ at fundamental discriminants $d$. The loss of powers of logarithm in part (1) is ultimately due to the fluctuation in the size of $|c_g(|d|)|$. Indeed, the study of sign changes relies on estimating the second and fourth moments of the Fourier coefficients and a consequence of the variation in the size of $|c_g(|d|)|$ is that these moments are not of the same order of magnitude as the fourth moment amplifies the large values of $|c_g(|d|)|$. The proof of the second part exploited the multiplicativity of $c_g(n)$ at squares following the arguments of \cite{Lester-Matomaki-Radziwill2018}. 

In the same article it was suggested that it might be possible to obtain a result that holds for a positive proportion of forms in $\mathcal S_K$ with $\gg K/Y$ real zeroes in $\mathcal F_Y$ by introducing a suitable mollifier to the argument (we shall review the approach of \cite{Jaasaari2026} in the next section). The goal of the present work is to verify this prediction. In the following section we introduce a mollifier that is suitable for our aim. The mollifier $M_g(d)$ is chosen so that typically $|c_g(|d|)|M_g(d)$ is of constant size.
Our main result is the following, which unifies and improves the results in Theorem \ref{prior-theorem}. 

\begin{theorem}\label{Main-theorem}
Let $K$ be a large parameter and $j\in\{1,2\}$. Then for $\gg K^2$ of the forms $g\in\mathcal S_K$ we have 
\[\#\{z\in\mathcal Z(g)\cap\delta_j\cap\mathcal F_Y\}\gg\frac K{Y}  \]
for $\sqrt{K\log K}\leq Y\leq K^{1-\vartheta}$ with any small fixed constant $\vartheta>0$. 
\end{theorem}

\begin{rem} (1) It might be possible to extract an explicit proportion of forms for which Theorem \ref{Main-theorem} holds from our argument, but this seems difficult and so we refrain from trying this. \\

\noindent (2) In the range $\sqrt{k\log k}\leq Y<\frac1{100}k$ it is expected that 
\[ 
\#\{z\in\mathcal Z(g)\cap\mathcal F_Y\}\asymp\frac kY
\]
for $g\in S_{k+\frac12}^+(4)$. 
\end{rem}
\noindent Theorem \ref{Main-theorem} should be compared to results of Lester, Matom\"aki, and Radziwi{\l\l} \cite{Lester-Matomaki-Radziwill2018} for integral weight forms. They showed that for any fixed $\eps>0$ there exists a subset\footnote{Throughout the paper we write $H_k$ for the set of holomorphic cusp forms of integral weight $k$ and full level. We also write $B_k$ for the Hecke eigenbasis of $H_{2k}$.} $S_k\subset B_k$, containing more than $(1-\eps)\#B_k$ elements, such that for every $f\in S_k$ we have
\[ 
\#\{z\in Z(f)\cap\delta_j\cap\mathcal D_Y\}\geq c(\eps)\#\{Z(f)\cap \mathcal D_Y\}\]
for each $j\in\{1,2\}$ provided that $\sqrt{k\log k}<Y<\delta(\eps)k$ and $k\longrightarrow\infty$ for some positive constants $c(\eps)$ and $\delta(\eps)$ depending only on $\eps$. Here of course $\mathcal D_Y:=\{z\in\mathcal D:\,\Im(z)\geq Y\}$. Furthermore, under the Generalised Lindel\"of Hypothesis they showed that 
\[ 
\#\{z\in Z(f)\cap\delta_j\cap \mathcal D_Y\}\gg_\eps \left(\frac kY\right)^{1-\eps}
\]
in the same range of $Y$.

The proof of their result relies on the breakthrough of Matom\"aki and Radziwi{\l\l} \cite{Matomaki-Radziwill2016} on multiplicative functions in short intervals. Concerning the first result mentioned, note that for every form $f$ we cannot do as well, even on the assumption of the Generalised Lindel\"of or Generalised Riemann Hypothesis (GRH). The reason is that in order to produce real zeroes of $f$ we look at sign changes of the Hecke eigenvalues $\lambda_f(m)$. In order to obtain a positive proportion of the zeroes on the line we need a positive proportion of sign changes between the coefficients of $\lambda_f(m)$ in appropriate ranges of $m$. However, we cannot e.g. rule out the scenario where for all $p\leq (\log k)^{2-\eps}$ we have $\lambda_f(p)=0$, even conditionally. 

Broadly the method used to prove Theorem \ref{Main-theorem} is the same as in our previous work \cite{Jaasaari2026}, which greatly differs from the arguments used in \cite{Ghosh-Sarnak2012, Matomaki2016, Lester-Matomaki-Radziwill2018}. However, computing the mollified moments of $c_g(|d|)$ is significantly more challenging compared to estimating the unmollified moments in \cite{Jaasaari2026}. In particular, the main novelty of the present work is the sharp estimation of the mollified fourth moment of the Fourier coefficients that requires relating the moment in question to a suitable random model. Finally, observe that our result requires considering forms in a larger family $\cup_{k\sim K}B_{k+1/2}^+$ instead of just $B_{k+1/2}^+$. The reason for this is that, similarly as in \cite{Jaasaari2026}, both the averages over the forms $g\in B_{k+1/2}^+$ and the weights $k$ are needed in order to evaluate the mollified moments of the Fourier coefficients sharply, see the discussion in \cite[Section 2]{Jaasaari2026}. 

\section{The Strategy}

\noindent In this section we explain the main ideas behind Theorem \ref{Main-theorem}. But before that we briefly outline the structure of the approach in \cite{Jaasaari2026}. There the main idea was to study sign changes of the Fourier coefficients $c_g(|d|)$ of a half-integral weight Hecke cusp form $g$ along the (odd) fundamental discriminants $d$. It is convenient to normalise the Fourier coefficients $c_g(|d|)$ by the factor $\sqrt{\alpha_g}$, where
\begin{align}\label{Normalisation}
\alpha_g:=\frac{\Gamma\left(k-\frac12\right)}{2(4\pi)^{k-\frac12}\|g\|_2^2}>0.
\end{align}

Throughout this section we set $X=K/Y$. Very briefly, it follows from a steepest descent argument that in order to obtain real zeroes of $g$ one needs to produce fundamental discriminants $d_\pm$ in short intervals $x\leq (-1)^kd_\pm\leq x+H$ with $H$ as small as possible in terms of $X\sim x$ so that $\sqrt{\alpha_g}c_g(|d_-|)<-k^{-\delta}<k^{-\delta}<\sqrt{\alpha_g}c_g(|d_+|)$ for some sufficiently small $\delta>1/2$. For the purpose of exposition we pretend here that we are only looking for numbers $d_\pm$ for which $c_g(|d_-|)<0<c_g(|d_+|)$. To detect such an event one considers the simple fact that 
\[ 
\left|\sumflat_{x\leq (-1)^kd\leq x+H}\sqrt{\alpha_g}c_g(|d|)\right|\leq\sumflat_{x\leq (-1)^kd\leq x+H}\sqrt{\alpha_g}|c_g(|d|)
\]
and notes that the inequality is strict if and only if $c_g(|d|)$ has a sign change in the interval $[x,x+H]$. Here $\sumflat\,$ signifies that we are summing over odd fundamental discriminants. In \cite{Jaasaari2026} it was shown that one is able to detect sign changes for many forms if $H$ is chosen to be a power of $\log K$. Indeed, choosing $H=(\log X)^9$ it was shown that for $\gg_\eps K^2/(\log K)^{3/2+\eps}$ of the forms $g\in S_K$ the inequality 
\[ 
\left|\sumflat_{x\leq (-1)^kd\leq x+H}\sqrt{\alpha_g}c_g(|d|)\right|<\sumflat_{x\leq (-1)^kd\leq x+H}\sqrt{\alpha_g}|c_g(|d|)
\]
holds for $\gg X/(\log X)^{5/2}$ of $x\sim X$ from which one easily deduces part (1) of Theorem \ref{prior-theorem}. The two auxiliary results needed for this argument are contained in the following lemma. Here $\omega_f$ (sometimes we write $\omega_g$ if $g\in S_{k+1/2}^+(4)$ corresponds to $f\in H_{2k}$ under the Shimura correspondence. Similarly for the mollifier we shall use $M_g(d)$ and $M_f(d)$ interchangeably) are the standard harmonic weights that appear in the Petersson trace formula. 

\begin{lemma}(\cite[Lemmas 2.7. and 2.8]{{Jaasaari2026}})\label{unmollified_moments}
Let $h$ and $\phi$ be compactly supported non-negative smooth weight functions on $\mathbb R_+$. Then for $X\ll\sqrt K$ we have that
\begin{enumerate}
\item \[ 
\sum_{k\in\Z}h\left(\frac{2k-1}K\right)\sum_{g\in B_{k+1/2}^+}\alpha_g\sumflat_d\,|c_g(|d|)|^2\phi\left(\frac{|d|}X\right)=\frac{2XK}{3\pi^2}\widehat h(0)\widehat\phi(0)+O_\eps\left(KX^{1/2+\eps}\right).
\]
\item 
\[ 
\sum_{k\in\Z}h\left(\frac{2k-1}K\right)\sum_{g\in B_{k+\frac12}^+}\alpha_g^2\omega_g^{-1}\sumflat_d\,|c_g(|d|)|^4\phi\left(\frac{|d|}X\right)\ll XK\log (XK).
\]
\end{enumerate}
\end{lemma}
\noindent The loss of powers of the logarithm in Theorem \ref{prior-theorem} stems from the fact that in the previous lemma the bounds for the second and fourth moments of the Fourier coefficients are not of the same order of magnitude. As mentioned above, the reason for this is that the size of $|c_g(|d|)|$ fluctuates as $d$ traverses over the fundamental discriminants. 

To remedy this one can introduce a positive quantity $M_g(d)$ with the property that $|c_g(|d|)|M_g(d)\approx \omega_g^{1/2}\alpha_g^{-1/2}$ and evaluate the moments of $c_g(|d|)M_g(d)$. This then yields information about the sign changes of $c_g(|d|)$ as $M_g(d)>0$. Below we shall choose the mollifier $M_g(d)$ so that, assuming $X\ll\sqrt{K}$, 
\begin{align*}
&\sum_{k\in\Z}h\left(\frac{2k-1}K\right)\sum_{g\in B_{k+1/2}^+}\alpha_g^2\omega_g^{-1}\sumflat_d\,|c_g(|d|)|^4M_g(d)^4\phi\left(\frac{|d|}X\right)\\
&\asymp\sum_{k\in\Z}h\left(\frac{2k-1}K\right)\sum_{g\in B_{k+1/2}^+}\alpha_g\sumflat_d\,|c_g(|d|)|^2M_g(d)^2\phi\left(\frac{|d|}X\right)\asymp XK
\end{align*}
for compactly supported smooth weight functions $h$ and $\phi$, which then implies Theorem \ref{Main-theorem} by repeating the arguments of \cite{Jaasaari2026} with $c_g(|d|)$ replaced by $c_g(|d|)M_g(|d|)$. 

We now compare the strategies to compute the mollified moments to the methods used to establish Lemma \ref{unmollified_moments} in \cite{Jaasaari2026}. Concerning the second moment, in \cite{Jaasaari2026} one executes the $g$-sum using the half-integral weight variant of the Petersson trace formula after which the $d$- and $k$-sums are evaluated by Poisson summation. For the mollified second moment we instead first use Waldspurger's formula to express $|c_g(|d|)|^2$ as a central $L$-value on which an approximate functional equation is then applied. At this point the sum over $f\in B_k$ is evaluated by the Petersson formula for integral weight forms. After that the $d$- and $k$-sums are evaluated by Poisson summation. The final step is to relate the resulting main term to an Euler product that is estimated by Mertens' theorem. 

The unmollified fourth moment was treated by starting with Waldspurger's formula and an approximate functional equation, after which a large sieve inequality of Deshouillers and Iwaniec was used to estimate the double $(g,k)$-sum. The final step was to average over the fundamental discriminants trivially. With a mollifier this strategy does not work and we need to proceed differently. However, the first steps are identical: by Waldspurger's formula we need to estimate the sum
\[ 
\sumflat_{d}\phi\left(\frac{|d|}X\right)\sum_{k\in\Z}h\left(\frac{2k-1}K\right)\sum_{f\in B_k}\omega_f L\left(\frac12,f\otimes\chi_d\right)^2M_f(d)^4.
\]
Using an approximate functional equation we express the central $L$-value as essentially a finite Dirichlet series. After expanding the definition of $M_f(d)$ we execute the sum over $B_k$ using the Petersson formula. This splits the sum into two parts, diagonal and off-diagonal, in a natural way. The diagonal term can be written as an expectation value
\begin{align}\label{sum-intro}
\sumflat_{d}\phi\left(\frac{|d|}X\right)\sum_{k\in\Z}h\left(\frac{2k-1}K\right)\mathbb E(L(X;d,k)M_2(X;d)),
\end{align}
where $L(X;d,k)$ is a random $L$-function built from random variables $X$, which model the behaviour of Hecke eigenvalues $\lambda_f(m)$. Similarly, $M_2(X;d)$ is a random mollifier modelling the behaviour of $M_f(d)^4$. Now, using Mellin inversion we may express (\ref{sum-intro}) as a double line integral
\[ 
\sumflat_{d}\phi\left(\frac{|d|}X\right)\sum_{k\in\Z}h\left(\frac{2k-1}K\right)\mathbb E\left(\frac1{(2\pi i)^2}\int\limits_{(1)}\int\limits_{(1)}\left(\frac{2\pi}{|d|}\right)^{-s_1-s_2}\frac{\Gamma(s_1+k)}{\Gamma(k)}\cdot\frac{\Gamma(s_2+k)}{\Gamma(k)}e^{s_1^2+s_2^2} F(s_1,s_2;X,d)\,\frac{\mathrm d s_1\,\mathrm d s_2}{s_1s_2}\right)
\]
for a certain function $F(s_1,s_2;X,d)$ for which $\mathbb E(F(s_1,s_2;X,d))$ has poles at $s_2=0$ and $s_2=-s_1$. It will turn out that the contribution of the latter pole cancels the main term of the off-diagonal from the initial application of the Petersson formula. This gives a new instance of a similar phenomenon appearing in some prior works \cite{Blomer2004, Blomer2008, Khan2010}. The computations here have similarities with the arguments of Khan \cite{Khan2010} and we are able to use some of the technical results he proved. However, there are also features that are not present in \cite{Khan2010}. In particular, the main term  from the off-diagonal arises in a very different manner compared to Khan's work. 

The pole at $s_2=0$ gives the dominant contribution, which turns out to be $\ll XK$ after some computations with Euler products. Here the presence of random variables greatly simplifies the computations, especially allowing us to control the effect of "small" primes in the mollifier using some ideas from \cite{Bui-Evans-Lester-Pratt2022}. Back-of-the-envelope calculation also suggests that the off-diagonal from the Petersson formula can be evaluated using arguments of Balkanova and Frolenkov \cite{Balkanova-Frolenkov2021}. This is based on the observation that complicated special functions that arise from an application of the (weighted) Petersson formula are solutions of certain second-order differential equations, and therefore their behaviour can be modelled by using the WKB-approximation that is widely used for example in quantum mechanics.

\subsection{Organisation of the article}

This paper is organised as follows. In Section $4$ we describe the mollifier $M_g(d)$ used in this work. In the following section we gather basic facts about half-integral weight modular forms and other auxiliary results we need. In Section $6$ we introduce a random model for the Hecke eigenvalues and construct a random mollifier from these random variables. In Section $7$ we evaluate a character sum that arises in the computation of the fourth moment of the Fourier coefficients. The second mollified moment is estimated in Section $8$. The estimation of the fourth moment occupies the next three sections. Finally the main result, Theorem \ref{Main-theorem}, is proved in Section $12$.

\subsection{Acknowledgements}

\noindent This work was supported by the Finnish Cultural Foundation. The author is grateful to Kaisa Matom\"aki for interesting discussions, and to Steve Lester for encouragement and suggesting that the methods of \cite{Bui-Evans-Lester-Pratt2022} might be useful when computing the mollified fourth moment of the Fourier coefficients. 

\section{Notations}

\noindent We use standard asymptotic notation. If $f$ and $g$ are complex-valued functions defined on some set, say $\mathcal G$, then we write $f\ll g$ to signify that $|f(x)|\leqslant C|g(x)|$ for all $x\in\mathcal G$ for some implicit constant $C\in\mathbb R_+$. The notation $O(g)$ denotes a quantity that is $\ll g$, and $f\asymp g$ means that both $f\ll g$ and $g\ll f$. We write $f=o(g)$ if $g$ never vanishes in $\mathcal G$ and $f(x)/g(x)\longrightarrow 0$ as $x\longrightarrow\infty$. Moreover, we write\footnote{This should not be confused with the notation $\ell\sim L$ explained in footnote $2$.} $f\sim g$ if $f(x)/g(x)\longrightarrow 1$ as $x\longrightarrow\infty$. The letter $\varepsilon$ denotes a positive real number, whose value can be fixed to be arbitrarily small, and whose value can be different in different instances in a proof.  All implicit constants are allowed to depend on $\varepsilon$, on the implicit constants appearing in the assumptions of theorem statements, and on anything that has been fixed. When necessary, we will use subscripts $\ll_{\alpha,\beta,...},O_{\alpha,\beta,...}$, etc. to indicate when implicit constants are allowed to depend on quantities $\alpha,\beta,...$

We define $\chi_d(\cdot):=\left(\frac d\cdot\right)$, the Kronecker symbol, for all non-zero odd integers $d$. 
Let us also write $1_{m=n}$ for the characteristic function of the event $m=n$ and $1_{[K,2K]}(x)$ for the characteristic function for $x\in[K,2K]$. Furthermore, $\Re(s)$ and $\Im(s)$ are the real- and imaginary parts of $s\in\mathbb C$, respectively, and occasionally we write $\sigma$ for $\Re(s)$. We write $e(x):=e^{2\pi ix}$. For a compactly supported smooth function $\phi$, we define its Fourier transform $\widehat\phi(y)$ by
\[
\widehat\phi(y):=\int\limits_\mathbb R \phi(x)e(-xy)\,\mathrm d x.
\]
 
The sum $\sum_{a\,(c)}^*$ means that the summation is over residue classes coprime to the modulus. We also write $\tau(\chi)$ for the quadratic Gauss sum associated to a character $\chi$. Given coprime integers $a$ and $c$, we write $\overline a\,(\text{mod } c)$ for the multiplicative inverse of $a$ modulo $c$. As usual, $\Gamma$ denotes the Gamma function and $\mu$ denotes the M\"obius function. In addition, $d(n)$ is the usual divisor function, $\sigma_k(n)$ is the generalised divisor function, and $\varphi(n)$ is Euler's totient function. The notation $n=\square$ means that a natural number $n$ is a perfect square. We also write $[a,b]$ for the least common multiple of two natural numbers $a$ and $b$, this should not be confused with a notation $[a,b]$ used for a real interval. The notation $p^\alpha\|n$ means that $p^\alpha|n$, but $p^{\alpha+1}\nmid n$. Of course $\zeta$ denotes the Riemann zeta function and we write its local factors as $\zeta_p(s):=(1-p^{-s})^{-1}$ so that $\zeta(s)=\prod_p\zeta_p(s)$ for $\Re(s)>1$. We also write $\text{Tr}(A)$ for the trace of a matrix $A$, $\text{rad}(n)$ for the radical of $n\in\N$, $|\mathcal X|$ for the measure of a set $\mathcal X$, and $\mathbb E(X)$ for the expectation value of a random variable $X$. Finally, $\sumflat\,\,$ means that we are summing over all odd fundamental discriminants and $\text{sgn}(d)$ denotes the sign of $d$. 
  
\section{Choosing the mollifier}

\subsection{Heuristics}

\noindent Let $d$ denote a fundamental discriminant, and $\chi_d(\cdot)=\left(\frac d\cdot\right)$ denote the primitive quadratic character of conductor $|d|$. Let $f\otimes\chi_d$ denote the twist of $f\in B_{k}$ by the character $\chi_d$, and let $L(s,f\otimes\chi_d)$ denote the twisted $L$-function
\[ 
L(s,f\otimes\chi_d):=\sum_{m=1}^\infty\frac{\lambda_f(m)\chi_d(m)}{m^s}=\prod_p L_p(s,f\otimes\chi_d)\]
for $\Re(s)>1$. Here 
\[ 
L_p(s,f\otimes\chi_d):=\left(1-\frac{\lambda_f(p)\chi_d(p)}{p^s}+\frac{\chi_d(p)^2}{p^{2s}}\right)^{-1}.
\]
\noindent The Dirichlet series $L(s,f\otimes\chi_d)$ extends analytically to the entire complex plane, and satisfies the functional equation
\[\Lambda(s,f\otimes\chi_d)=i^{2k}\varepsilon(d)\Lambda(1-s,f\otimes\chi_d),
\]
where
\[
\Lambda(s,f\otimes\chi_d):=\left(\frac{|d|}{2\pi}\right)^s\Gamma\left(s+\frac{2k-1}2\right)L(s,f\otimes\chi_d),
\]
and $\varepsilon(d)=\left(\frac d{-1}\right)$ is $1$ or $-1$ depending on whether $d$ is positive or negative.  

By Waldspurger's formula the absolute square of $c_g(|d|)$ is proportional to the central $L$-value $L(1/2,f\otimes \chi_d)$ when $(-1)^kd>0$, where $f\in H_{2k}$ is a classical holomorphic cusp form attached to $g\in S_{k+\frac 12}^+(4)$ via the Shimura correspondence. This motivates us to choose the mollifier $M_g(d)$ so that it behaves like $L(1/2,f\otimes\chi_d)^{-1/2}$, at least for typical $g$ and $d$.
 
For $|d|\gg k^\eps$ we expect that $\log L(1/2,f\otimes\chi_d)$ to be approximated by
\[ 
\sum_{p^\ell< |d|^\eps}\frac{\left(\alpha_p^\ell+\beta_p^\ell\right)\chi_d(p)^\ell}{\ell p^{\ell/2}},\]
where $\alpha_p,\beta_p$ are the Satake parameters of $f$ at $p$. Recalling that $\alpha_p+\beta_p=\lambda_f(p)$ and $\alpha_p^2+\beta_p^2=\lambda_f(p^2)-1$ it is expected (and can be shown under GRH) that the sum above is
\[ 
\sum_{p<|d|^\eps}\frac{\lambda_f(p)\chi_d(p)}{\sqrt p}-\frac12\log\log|d|\left(1+o(1)\right).\]
Thus heuristically we expect (unconditionally) that
\begin{align}\label{L-function-approximation}
L\left(\frac12,f\otimes\chi_d\right)\approx(\log |d|)^{-1/2}\exp\left(\sum_{p<|d|^\eps}\frac{\lambda_f(p)\chi_d(p)}{\sqrt p}\right).
\end{align} 
Here the approximation sign $\approx$ should not be taken too literally. Recall that we wish to choose the mollifier so that it approximates $L(1/2,f\otimes\chi_d)^{-1/2}$. By the discussion above a natural choice for the mollifier would be, for suitable $x$ depending on $X$ when $d\sim X$,
\[ 
M_g(d)=(\log x)^{1/4}\exp(P(f,d)),\]
where 
\[
P(f,d):=-\frac12\sum_{p\leq x}\frac{\lambda_f(p)\chi_d(p)}{\sqrt p}.
\] 
In order for $M_g(d)$ to be a good approximation for $L(1/2,f\otimes\chi_d)^{-1/2}$ one wishes to choose $x\asymp X^\eps$ for some fixed $\eps>0$. However, with this choice one faces a technical problem that expanding the exponential factor into a Taylor series leads to a long Dirichlet polynomial, which in turn makes it impossible to facilitate certain necessary computations. On the other hand, for rather small $x$, say $x\asymp\log X$, the resulting Dirichlet polynomial has a short length, but in this case we do not expect $M_g(d)$ to be a good approximation for $L(1/2,f\otimes\chi_d)^{-1/2}$.

The idea that allows us to take $x$ to be a small power of $X$ is to use an iterative scheme of Radziwi{\l\l} and Soundararajan \cite{Radziwill-Soundararajan2015}, which is partly motivated by the Brun--Hooley sieve. Heuristically we expect that $P(f,d)$ has a Gaussian limiting distribution over $d\sim x$ with mean zero and variance $\sim\log\log x$. We consider disjoint sets of primes in the intervals $I_0,I_1,...,I_J$ with $I_j=[K^{\theta_{j-1}},K^{\theta_j}[$ for $j=1,...,J$, and $I_0:=]c,K^{\theta_0}]$, where $c>2$ is fixed and chosen to be sufficiently large. Due to the expected Gaussian distribution, choosing $\theta_j$'s appropriately, we anticipate that each of the terms
\[ 
\exp\left(-\frac12\sum_{p\in I_j}\frac{\lambda_f(p)\chi_d(p)}{\sqrt p}\right)
\]
should typically be approximated by a short Dirichlet polynomial depending on the size of $\theta_j$ as the sum inside the argument is expected to be small in absolute value for most $d$. We take $J$ to be so large that $\theta_J\geq\eta_2$ for some fixed $\eta_2>0$, meaning that $J\asymp\log\log\log K$. This leads one to expect that for a typical $d\sim X$ we should have that
\[ 
\exp\left(-\frac12\sum_{c<p\leq X^\eps}\frac{\lambda_f(p)\chi_d(p)}{\sqrt p}\right)\approx\prod_{j=0}^J\exp\left(-\frac12\sum_{p\in I_j}\frac{\lambda_f(p)\chi_d(p)}{\sqrt p}\right)
\] 
can be approximated by a short Dirichlet polynomial. Due to technical reasons it is better to replace $\lambda_f(m)$ with the completely multiplicative function $a_f(m)$ defined by $a_f(p):=\lambda_f(p)$ at the primes. The key observation is that using the Taylor expansion\footnote{More precisely, in the form saying that $e^z=(1+O(e^{-9V}))\sum_{0\leq j\leq 10V}z^j/j!$ for $|z|\leq V$.} and multinomial theorem we have
\begin{align*}
\exp\left(-\frac12\sum_{p\in I_j}\frac{a_f(p)\chi_d(p)}{\sqrt p}\right)&\approx \sum_{0\leq \ell\leq 10\ell_j}\frac1{\ell!}\left(-\frac12\sum_{p\in I_j}\frac{a_f(p)\chi_d(p)}{\sqrt p}\right)^\ell\\
&=\sum_{\substack{p|n\Rightarrow p\in I_j\\
\Omega(n)\leq 10\ell_j}}\frac{\lambda(n)a_f(n)\chi_d(n)\nu(n)}{2^{\Omega(n)}\sqrt n}
\end{align*}
if 
\[ 
\left|\sum_{p\in I_j}\frac{a_f(p)\chi_d(p)}{\sqrt p}\right|\leq\ell_j.
\]
Here $\Omega(n)$ is the number of prime divisors of a positive integer $n$ (counted with multiplicity), $\nu$ is a multiplicative function such that $\nu(p^\alpha)=1/\alpha !$, and $\lambda(n):=(-1)^{\Omega(n)}$ is the Liouville function.

A concrete choice of parameters that suffices for us is the same one as in \cite{Lester-Radziwill2021} (see also \cite{Bui-Evans-Lester-Pratt2022} for a similar choice in a slightly different context), namely
\begin{align*}
\theta_j:=\eta_1\frac{e^j}{(\log\log K)^5} \qquad \text{and }\qquad \ell_j:=2\lfloor \theta_j^{-3/4}\rfloor,
\end{align*}
where $\eta_1>0$ is a sufficiently large absolute constant. We choose $J$ so that $\eta_2\leq\theta_J\leq e\eta_2$ for sufficiently small constant $\eta_2>0$. 

The discussion above motivates us to choose our mollifier as
\begin{align}\label{mollifier}
M_g(d):=(\log K)^{1/4}\prod_{j=0}^J M_g(d;j), 
\end{align}
where 
\[ 
M_g(d;j):=\sum_{\substack{p|n\Rightarrow p\in I_j\\
\Omega(n)\leq\ell_j}}\frac{a_f(n)\lambda(n)\nu(n)\chi_d(n)}{2^{\Omega(n)}\sqrt n}.\]
This mollifier is designed so that it behaves like $L(1/2,f\otimes \chi_d)^{-1/2}$ for typical $f\in B_{k}$ and fundamental discriminant $d$. Note that as $c>2$ it follows that $M_g(d;j)$ is supported only on odd $n$'s. Throughout the paper we set 
\[ 
\delta_0:=\sum_{j=0}^J\ell_j\theta_j
\]
and note that $M_g(d)\ll K^{\delta_0}(\log K)^{1/4}$. Furthermore, $M_g(d)>0$ as noted in Section $4$ of \cite{Lester-Radziwill2021}. By choosing $\eta_2$ to be small enough we can guarantee that $\delta_0$ is sufficiently small for certain estimates later. This will play an important role in many parts of our argument. 

The method used to prove Theorem \ref{Main-theorem} relies crucially on the estimation of certain mollified moments of quadratic twists of modular $L$-functions. The relevant results are the content of the following two propositions, which replace Lemmas 2.7. and 2.8. in \cite{Jaasaari2026}.

\begin{prop}\label{second-moment-with-average}
Let $\phi$ and $h$ be non-negative smooth compactly supported functions on $\mathbb R_+$. Suppose that $X\ll\sqrt K$. Then we have 
\[\sum_{k\in\Z}h\left(\frac{2k-1}K\right)\sum_{g\in B^+_{k+\frac12}}\alpha_g\sumflat_d\, |c_g(|d|)|^2 M_g(d)^2\phi\left(\frac{|d|}X\right)\asymp XK. \]
\end{prop}

\begin{prop}\label{fourth-moment-with-average}
Let $\phi$ and $h$ be non-negative smooth compactly supported functions on $\mathbb R_+$. Suppose that $X\ll\sqrt K$. Then we have
\begin{align}\label{Fourth-moment}
&\sum_{k\in\Z}h\left(\frac{2k-1}K\right)\sum_{g\in B^+_{k+\frac12}}\alpha_g^2\omega_g^{-1}\sumflat_{d}\, |c_g(|d|)|^4 M_g(d)^4\phi\left(\frac{|d|}X\right)\ll XK. 
\end{align}
\end{prop}

\noindent In the latter proposition and throughout the paper we write as an abuse of notation $\omega_g$ for $\omega_f$ when $f$ is the Shimura lift of $g$. 

Given these the proof of Theorem \ref{Main-theorem} can be completed similarly as in \cite{Jaasaari2026}. This deduction is achieved in Section $12$. 

\subsection{Evaluating the mollified fourth moment}

\noindent The proof of Proposition \ref{fourth-moment-with-average} warrants some comments. But before that, one easily observes using the complete multiplicativity that for any $\ell\in\mathbb N$,
\begin{align}\label{Mollifier-power}
M_g(d)^{2\ell}=(\log K)^{\ell/2}\sum_{n\leq K^{2\ell\delta_0}}\frac{h_\ell(n)a_f(n)\lambda(n)\chi_d(n)}{2^{\Omega(n)}\sqrt n}.
\end{align}
Here
\[h_\ell(n):=\sum_{\substack{n_0\cdots n_J=n \\
p|n_j\Rightarrow p\in I_j\,\forall\, 0\leq j\leq J\\
\Omega(n_j)\leq 2\ell\ell_j\,\forall\, 0\leq j\leq J}}\nu_{2\ell}(n_0;\ell_0)\cdots\nu_{2\ell}(n_J;\ell_J) \]
and 
\[
\nu_r(n;\ell):=\sum_{\substack{n_1\cdots n_r=n \\
\Omega(n_j)\leq \ell\,\forall\, 1\leq j\leq r }}\nu(n_1)\cdots\nu(n_r).
\]
Let $\nu_j(n):=(\nu*\cdots*\nu)(n)$ denote the $j$-fold convolution of $\nu$. Note that $\nu_r(n;\ell)\leq\nu_r(n)$ and that if $\Omega(n)\leq\ell$, then $\nu_r(n;\ell)=\nu_r(n)$. 

After opening $M_g(d)^4$ using (\ref{Mollifier-power}) and applying Waldspurger's formula to express $|c_g(|d|)|^2$ as a central value of an $L$-function we see that in order to prove Proposition \ref{fourth-moment-with-average} we would require an asymptotic evaluation of the multiple average
\[ 
\sum_{k\in\Z}h\left(\frac{2k-1}K\right)\sum_{f\in B_k}\omega_f\sumflat_d\, L\left(\frac12,f\otimes\chi_d\right)^2 a_f(r)\chi_d(r)\phi\left(\frac{|d|}X\right)
\]
for $r\leq K^{4\delta_0}$.

This can be done, but applying the resulting formula to the left-hand side of (\ref{Fourth-moment}) results in an unwieldy expression for the main term, which is hard to evaluate. Instead we introduce a random $L$-function in which the Hecke eigenvalues $\lambda_f(m)$ are modelled by random variables $X(m)$ (defined shortly), and then match our expression with the random analogue. Comparison with a random model allows us to sidestep a number of technical points that would otherwise require involved effort to resolve. In particular, using the independence of the random variables $\{X(p)\}_p$ at the primes reduces many of the computations to "local" ones evaluated at each prime.
 
\section{Preliminaries}

\subsection{Half-integral weight forms}
 
Let $S_{k+\frac12}(4)$ denote the space of holomorphic cusp forms of weight $k+\frac12$ for the Hecke congruence group $\Gamma_0(4)$. Any such form $g$ has a Fourier expansion of the form
\begin{align}\label{Fourier-exp}
g(z)=\sum_{n=1}^\infty c_g(n)n^{\frac k2-\frac14}e(nz),
\end{align}
where $c_g(n)$ are the Fourier coefficients of $g$.

For $g,h\in S_{k+\frac12}(4)$, we define the Petersson inner product $\langle g,h\rangle$ to be
\begin{align}\label{Inner_product}
\langle g,h\rangle:=\int\limits_{\Gamma_0(4)\backslash\mathbb H}g(z)\overline{ h(z)}y^{k+\frac12}\frac{\mathrm d x\,\mathrm d y}{y^2}.
\end{align}

\noindent For any odd prime $p$ there exists a Hecke operator $T(p^2)$ acting on the space of half-integral weight modular forms. We call a half-integral weight cusp form a Hecke cusp form if $T(p^2)g=\gamma_g(p)g$ for all $p>2$ for some $\gamma_g(p)\in\mathbb C$. 

The Kohnen plus subspace $S_{k+\frac12}^+(4)\subset S_{k+\frac12}(4)$ consists of all weight $k+\frac12$ Hecke cusp forms whose $n^{\text{th}}$ Fourier coefficient vanishes whenever $(-1)^k n\equiv 2,3\,(\text{mod }4)$. This space has a basis consisting of simultaneous eigenfunctions of $T(p^2)$ for odd $p$. It is well-known that, as $k\longrightarrow\infty$, asymptotically one third of half-integral weight cusp forms lie in the Kohnen plus subspace. 

Kohnen proved \cite{Kohnen1982} that there exists a Hecke algebra isomorphism between $S^+_{k+\frac12}(4)$ and the space of level $1$ cusp forms of weight $2k$. That is, $S^+_{k+\frac12}(4)\simeq S_{2k}$ as Hecke modules. Also recall that every Hecke cusp form $g\in S^+_{k+\frac12}(4)$ can be normalised so that it has real Fourier coefficients and throughout the article we assume that $g$ has been normalised in this way.

The proof of our main result uses the explicit form of Waldspurger's formula due to Kohnen and Zagier \cite{Kohnen-Zagier1981}.

\begin{lemma} 
For a Hecke cusp form $g\in S^+_{k+\frac12}(4)$ we have
\begin{align}\label{Waldspurger-rel}
|c_g(|d|)|^2=L\left(\frac12,f\otimes\chi_d\right)\cdot\frac{(k-1)!}{\pi^k}\cdot\frac{\langle g,g\rangle}{\langle f,f\rangle}
\end{align}
for each fundamental discriminant $d$ with $(-1)^kd>0$, where $f$ is a holomorphic modular form attached to $g$ via the Shimura correspondence, normalised so that $\lambda_f(1)=1$. 
\end{lemma}

\begin{rem}
Using the normalisation (\ref{Normalisation}) the above formula can be written in the form
\begin{align}\label{W-rel} 
\alpha_g|c_g(|d|)|^2=\omega_fL\left(\frac12,f\otimes\chi_d\right).
\end{align}
\end{rem}

We also remark that $L(1/2,f\otimes\chi_d)$ vanishes when $(-1)^kd<0$ due to the sign in the functional equation. It follows directly from (\ref{Waldspurger-rel}) that $L(1/2,f\otimes\chi_d)\geq 0$ otherwise. Here for $f_1,f_2\in S_{2k}$ the Petersson inner product, which is still denoted by $\langle f_1,f_2\rangle$, is defined to be
\[ 
\langle f_1,f_2\rangle:=\int\limits_{\mathrm{SL}_2(\mathbb Z)\backslash\mathbb H}f_1(z)\overline{f_2(z)}y^{2k}\frac{\mathrm d x\,\mathrm d y}{y^2}.\]

\noindent Using the Hecke relations we have for any $\alpha\in\N$ and prime $p$ that
\[
\lambda_f\left(p\right)^\alpha=\sum_{c=0}^\alpha h_\alpha(c)\lambda_f\left(p^c\right), 
\]
where $h_\alpha(c)$ are non-negative integers given by
\[ 
h_\alpha(c):=\frac2\pi\int\limits_0^\pi(2\cos\theta)^\alpha\sin((c+1)\theta)\sin\theta\,\mathrm d \theta.
\]
Recall that $a_f(n)=\prod_{p^\alpha||n}\lambda_f(p)^\alpha$. By the repeated use of Hecke relations it is easy to see that 
\begin{align}\label{a-function}
a_f(n)=\sum_{u|n}c_n(u)\lambda_f(u),
\end{align}
where 
\[ 
c_n(u):=\prod_{\substack{p^c||u \\
p^\alpha||n}}h_\alpha(c).
\]
Later in the paper we require some specific values of $c_n(u)$ with $n$ (and $u$) powers of a fixed prime. These are easily computed and listed here:
\begin{center}
\begin{tabular}{ | m{3cm} | m{3cm} | m{3cm} | m{3cm} | m{3cm} |} 
  \hline
  $c_1(1)=h_0(0)=1$ & $c_p(1)=h_1(0)=0$ & $c_p(p)=h_1(1)=1$ & $c_{p^2}(1)=h_2(0)=1$ & $c_{p^2}(p)=h_2(1)=0$  \\ 
  \hline
  $c_{p^2}(p^2)=h_2(2)=1$ & $c_{p^3}(1)=h_3(0)=0$ & $c_{p^3}(p)=h_3(1)=2$ & $c_{p^3}(p^2)=h_3(2)=0$ & $c_{p^3}(p^3)=h_3(3)=1$\\ 
  \hline 
  \end{tabular}
  \medskip
Table 1. Certain values of $c_n(u)$.
\end{center}
\noindent We also note the easy bounds $0\leq h_\alpha(c)\leq 2^{\alpha}$ from which we immediately deduce the bounds $0\leq c_n(u)\leq 2^{\Omega(n)}$. 

An approximate functional equation also plays a key role in our work. The following is an easy modification of \cite[Lemma 5]{Radziwill-Soundararajan2015}. Let us set $\varepsilon_{k,d}:=(-1)^k\text{sgn}(d)$, where $\text{sgn}(d)$ denotes the sign of $d$.
\begin{lemma}\label{lem:AFE}
Let $f$ be a Hecke cusp form of weight $2k$ for the full modular group $\mathrm{SL}_2(\mathbb Z)$. Then for any fundamental discriminant $d$ we have that
\begin{align*}
L\left(\frac12,f\otimes\chi_d\right)=(1+\varepsilon_{k,d})\sum_{m=1}^\infty\frac{\lambda_f(m)\chi_d(m)}{\sqrt m}V_k\left(\frac m{|d|}\right),
\end{align*}
where, for any $\sigma>1/2$,
\begin{align}\label{integral-rep-for-weight}
V_k(x):=\frac1{2\pi i}\int\limits_{(\sigma)}\eta(s)x^{-s}e^{s^2}\frac{\mathrm d s}s\quad\text{with}\quad \eta(s):=(2\pi)^{-s}\frac{\Gamma(s+k)}{\Gamma(k)}.
\end{align}
Furthermore, we have that
\begin{align}\label{weight-function-asymptotics}
V_k(\xi)=1+O\left(\left(\frac\xi k\right)^{1/2-o(1)}\right)
\end{align}
as $\xi\longrightarrow 0$.

We also have the estimates
\begin{align}\label{der-bounds}
& V_k(\xi)\ll_A \left(\frac k\xi\right)^A,\nonumber \\
&V_k^{(B)}(\xi)\ll_{A,B}\xi^{-B}\left(\frac k\xi\right)^A
\end{align}
for any $A>0$ and integer $B\geq 0$.
\end{lemma}
\noindent In addition, using Stirling's formula there exists a holomorphic function $R(s,k)$ so that for $\Re(s)\geq -k/2$ we have $R(s,k)\ll |s|^2/k$ and
\begin{align}\label{Stirling}
\frac{\Gamma(s+k)}{\Gamma(k)}=k^s e^{R(s,k)}\left(1+O\left(k^{-1}\right)\right).
\end{align}
In particular, for $|\Re(s)|\leq\sqrt k$ we have
\[ 
\frac{\Gamma(s+k)}{\Gamma(k)}=k^s\left(1+O\left(\frac{|s|^2}k\right)\right).
\]
Note also that $V_k=V_{Ku+1}$ for $u=(k-1)/K$. 

To see that (\ref{Stirling}) holds, note that for $|\text{arg}(s)|\le \pi-\delta$, Stirling's formula gives
\[
\Gamma(s)=\sqrt{\frac{2\pi}{s}} \bigg(\frac{s}{e}\bigg)^s \bigg(1+O\bigg(\frac{1}{|s|}\bigg)\bigg),
\]
where the implied constant depends at most on $\delta$. Hence for $\Re(s) \geq -k/2$ we have that
\[
\begin{split}
\frac{\Gamma(k+s)}{\Gamma(k)}=&\bigg(1+O\bigg( \frac{1}{k} \bigg) \bigg) \sqrt{\frac{k}{k+s}} \bigg(\frac{k}{e} \bigg)^{-k} \bigg(\frac{k+s}{e} \bigg)^{k+s}\\
=&\bigg(1+O\bigg(\frac{1}{k} \bigg) \bigg) k^s e^{-s} \bigg(1+\frac{s}{k} \bigg)^{k+s-\frac12}.
\end{split}
\]
For $\Re(z)>-1/2$, we have that $\log (1+z)=z+O(|z|^2)$. Hence, for $\Re(s)\geq-k/2$ we have that
\[
\begin{split}
\bigg(1+\frac{s}{k} \bigg)^{k+s-\frac12}=&\exp\bigg((k+s-\tfrac12)\bigg(\frac{s}{k}+O\bigg( \frac{|s|^2}{k^2} \bigg) \bigg) \\
=& \exp\bigg(s+O\bigg( \frac{|s|^2}{k^2} \bigg) \bigg).
\end{split}
\]
Thus we conclude that for $\Re(s) \geq -k/2$ we have that
\[
\frac{\Gamma(k+s)}{\Gamma(k)}=k^s\exp\bigg( O\bigg( \frac{|s|^2}{k^2}\bigg)\bigg)\bigg(1+O\bigg(  \frac{1}{k} \bigg)\bigg),
\]
as desired.

\subsection{Basic tools}

\subsubsection{Summation formulas}

One of the most important tools is the Petersson trace formula. 

\begin{lemma}\label{Petersson}
Let $m$ and $n$ be natural numbers, and $k$ be a positive integer. Then
\[\sum_{f\in B_k}\omega_f\lambda_f(m)\lambda_f(n)=1_{m=n}+2\pi i^{2k}\sum_{c=1}^\infty\frac{S(m,n;c)}c J_{2k-1}\left(\frac{4\pi\sqrt{mn}}c\right),\]
where $S(m,n;c)$ is the usual Kloosterman sum and $J_\nu$ is the $J$-Bessel function. 
\end{lemma}

Let us now define two integral transforms. For a smooth compactly supported function $h$, we set 
\begin{align*}
\hbar(y):=\int\limits_0^\infty\frac{h(\sqrt u)}{\sqrt{2\pi u}}e^{iyu}\,\mathrm d u
\end{align*}
and
\[W_K^{(2)}(m,n,v):=\int\limits_0^\infty \frac{V_{\sqrt u K+1}(m)V_{\sqrt u K+1}(n)h(\sqrt u)}{\sqrt{2\pi u}}e^{iuv}\,\mathrm d u.\]
Main properties of these integral transforms have been worked out by Khan\footnote{In \cite{Khan2010} the definition of $V_k(\xi)$ differs slightly from ours, but identical arguments to those in \cite{Khan2010} lead to same bounds for expressions involving $V_k(\xi)$ in our situation.}. He has shown \cite[(2.17)]{Khan2010} that
\begin{align}\label{W_K-2-bound}
W_K^{(2)}(n,m,v)\ll_{A_1,A_2,B}\left(\frac Kn\right)^{A_1}\left(\frac Km\right)^{A_2}v^{-B}
\end{align}
for any $A_1,A_2>0$ and $B\geq 0$. Thus $W_K^{(2)}(n,m,v)$ is essentially supported on $n\leq K^{1+\varepsilon}$, $m\leq K^{1+\varepsilon}$ and $v\leq K^\varepsilon$. Moreover, using integration by parts we have estimates for the derivatives;
\begin{align}\label{W2-der}
\frac{\partial^j}{\partial\xi^j}W_K^{(2)}\left(\frac{\xi}{|d|},\frac z{|d|},\frac {K^2c}{8\pi\sqrt{\xi z}}\right)\ll_{j,A_1,A_2,B,\eps}\left(\frac{|d|K}{\xi}\right)^{A_1}\xi^{-j}\left(\frac{|d|K}z\right)^{A_2}\left(\frac{\sqrt{\xi z}}{K^2c}\right)^B K^\eps
\end{align}
for any $j\geq 0$, $A_1,A_2,B>0$, and $\eps>0$.


We also have the identity \cite[(2.14)]{Khan2010}
\begin{align}\label{J-Bessel-avg-II}
&2\sum_{k\in\Z}i^{2k} h\left(\frac{2k-1}K\right)V_k(n)V_k(m)J_{2k-1}(t) \nonumber \\
&=-\frac K{\sqrt t}\Im\left(e^{-2\pi i/8}e^{it}W_K^{(2)}\left(n,m,\frac{K^2}{2t}\right)\right)+O\left(\frac t{K^4}\int\limits_{\mathbb R}v^4\left| \int\limits_0^\infty V_{uK+1}(n)V_{uK+1}(m)h(u)e^{iuv}\,\mathrm d u\right|\,\mathrm d v\right).
\end{align}
\noindent By repeated integration by parts we see that the error term is 
\begin{align}\label{error-bound}
O_{A_1,A_2}\left(\frac t{K^4}\left(\frac Km\right)^{A_1}\left(\frac Kn\right)^{A_2}\right)
\end{align}
for any $A_1,A_2>0$.

Similarly using \cite[Lemma 5.8.]{Iwaniec1997} we have 
\begin{align}\label{J-Bessel-avg-III}
&2\sum_{k\in\Z} h\left(\frac{2k-1}K\right)V_k(n)V_k(m)J_{2k-1}(t) \nonumber \\
&=h\left(\frac tK\right)V_{\frac{t+1}2}(m)V_{\frac{t+1}2}(n)+O_{A,\eps}\left(K^{-2+\eps}\left(\frac Km\right)^{A_1}\left(\frac Kn\right)^{A_2}\right)
\end{align}
for any $A_1,A_2\geq 1$ and $\eps>0$.

A simple application of Poisson summation gives
\begin{align}\label{simple-Poisson}
\sum_{k\in\Z}h\left(\frac{2k-1}K\right)=\frac K2\widehat h(0)+O_B(K^{-B})
\end{align}
for every $B>0$. 

For real $\xi_1>0$ and $\xi_2>0$ we define
\[W(\xi_1,\xi_2,v):=\frac1{(2\pi i)^2}\int\limits_{(A_1)}\int\limits_{(A_2)}(2\pi )^{-x-y}e^{x^2+y^2}\xi_1^{-x}\xi_2^{-y}\hbar_{x+y}(v)\,\frac{\mathrm d x\,\mathrm d y}{xy},\]
where 
\[
\hbar_z(v):=\int\limits_0^\infty \frac{h(\sqrt u)}{\sqrt{2\pi u}}u^{z/2}e^{iuv}\,\mathrm d u. 
\]
Then using (\ref{Stirling}) it is easy to see that
\begin{align}\label{new-weight-function}
W_K^{(2)}(m,n,v)=W\left(\frac mK,\frac nK,v\right)+O_\eps\left(K^{-1+\varepsilon}\right). 
\end{align}
By integration by parts, we also have the bound
\begin{align}\label{derivative-bound-h}
\hbar_z^{(j)}(v)\ll_{\Re(z),j,B}(1+|z|)^B|v|^{-B}
\end{align}
for any integers $j,B\geq 0$, and consequently 
\[W(\xi_1,\xi_2,v)\ll_{B,A_1,A_2}\xi_1^{-A_1}\xi_2^{-A_2}v^{-B}\]
for any $A_1,A_2>0$ and $B\geq 0$. 

Let 
\[
\widetilde \hbar_z(s):=\int\limits_0^\infty v^{s-1}\hbar_z(v)\,\mathrm d v
\]
denote the Mellin transform of $\hbar_z$. This converges for $\Re(s)>0$ by (\ref{derivative-bound-h}) and is holomorphic in this half-plane. In particular, by Mellin inversion we have
\begin{align}\label{Mellin-for-h}
\hbar_z(v)=\frac1{2\pi i}\int\limits_{(c)} v^{-s}\widetilde\hbar_z(s)\,\mathrm d s 
\end{align}
for any $c>0$.

The treatment of the off-diagonal in the fourth moment computation requires the following auxiliary result \cite[Lemma 3.3.]{Khan2010} concerning properties of the Mellin transform of the function $\hbar_z$.
\begin{lemma}\label{Mellin-transform-lemma}
For $0<\Re(s)<1$, we have
\begin{align}\label{Mellin-transform-formula}\widetilde\hbar_z(s)=\int\limits_0^\infty\frac{h(\sqrt u)}{\sqrt{2\pi u}}u^{z/2-s}\Gamma(s)\left(\cos\left(\frac{\pi s}2\right)+i\sin\left(\frac{\pi s}2\right)\right)\,\mathrm d u
\end{align}
and the bound $\widetilde\hbar_z(s)\ll_{\Re(z)}(1+|z|)^3|s|^{-2}$.
\end{lemma}
\noindent Observe that the right-hand side of (\ref{Mellin-transform-formula}) is holomorphic for $\Re(s)>0$. Hence, the identity (\ref{Mellin-transform-formula}) continues to hold in this region as does the same bound for $\widetilde\hbar_z(s)$ uniformly in vertical strips, which can be seen by repeated integration by parts and using Stirling's formula. 

Another key tool is the Poisson summation formula.

\begin{lemma}
Let $f$ be a Schwartz function and $a$ be a residue class modulo $c$. Then
\[\sum_{n\equiv\,a\,(c)}f(n)=\frac1c\sum_n\widehat f\left(\frac nc\right)e\left(\frac{an}c\right),\]
where $\widehat f$ denotes the Fourier transform of $f$. Note that this reduces to the classical Poisson summation formula when $c=1$. 
\end{lemma}

\noindent We shall also need a different variant of the Poisson summation formula. For this, let us define, for any $\ell\in\mathbb Z$, a Gauss-type sum
\begin{align}\label{Gauss-sum} 
\tau_\ell(n):=\sum_{b\,(n)}\left(\frac nb\right)e\left(\frac{\ell b}n\right)=\left(\frac{(1+i)}2+\left(\frac {-1}n\right)\frac{(1-i)}2\right)G_\ell(n),
\end{align}
where
\[ G_\ell(n):=\left(\frac{(1-i)}2+\left(\frac {-1}n\right)\frac{(1+i)}2\right)\sum_{a\,(n)}\left(\frac na\right)e\left(\frac{a\ell}n\right).\]
It can be shown that for odd coprime $m,n$ one has $G_\ell(mn)=G_\ell(m)G_\ell(n)$ and if $p^\alpha\|\ell$ for a prime $p$, then the exact formula for $G_\ell\left(p^\beta\right)$ is given in \cite[Lemma 2.3]{Soundararajan2000}. The variant of the Poisson summation formula needed for our purposes is contained in the following lemma.

\begin{lemma}\label{Poisson-w-char} (\cite[Lemma 7.]{Radziwill-Soundararajan2015})
Let $n$ be an odd integer and $q$ positive integer so that $(n,q)=1$. Suppose that $F$ is a smooth and compactly supported function on $\mathbb R$. Finally, let $\eta$ be a reduced residue class modulo $q$. Then
\[\sum_{d\equiv\eta\,(q)}\left(\frac dn\right)F(d)=\frac1{qn}\left(\frac qn\right)\sum_{\ell\in\mathbb Z}\widehat F\left(\frac \ell{nq}\right)e\left(\frac{\ell\eta\overline n}q\right)\tau_\ell(n),\]
where $\widehat F$ is the usual Fourier transform.
\end{lemma}
\noindent The Gauss-type sum in the previous lemma may be evaluated explicitly.

\begin{lemma}\label{Gauss-type-sum} (\cite[Lemma 2.3.]{Soundararajan2000})
If $m$ and $n$ are coprime odd integers, then $\tau_\ell(mn)=\tau_\ell(m)\tau_\ell(n)$. Furthermore, if $p^\alpha$ is the highest power of $p$ that divides $\ell$ (setting $\alpha=\infty$ if $\ell=0$), then
\begin{align}
\tau_\ell\left(p^\beta\right)=\begin{cases}
0 & \text{if }\beta\leq\alpha\text{ is odd} \\
\varphi(p^\beta) & \text{if }\beta\leq\alpha\text{ is even} \\
-p^\alpha & \text{if }\beta=\alpha+1\text{ is even} \\
\left(\frac{\ell p^{-\alpha}}p\right)p^\alpha\sqrt p & \text{if }\beta=\alpha+1 \text{ is odd} \\
0 & \text{if }\beta\geq\alpha+2
\end{cases}
\end{align} 
\end{lemma}

\subsection{Other tools}

\noindent We also record the following well-known uniform estimate for the $J$-Bessel function \cite[(2.11'')]{Iwaniec-Luo-Sarnak2000}. For $\nu\geq 0$ and $x>0$, the $J_\nu$-Bessel function satisfies the bound 
\begin{align}\label{J-Bessel}
J_\nu(x)\ll\frac x{\sqrt{\nu+1}}\left(\frac{ex}{2\nu+1}\right)^\nu.
\end{align}
Another crucial auxiliary result is the stationary phase method for estimating oscillatory exponential integrals. We quote the following special case of a result \cite[Proposition 8.2.]{Blomer-Khan-Young2013} by Blomer, Khan, and Young that is uniform with respect to multiple parameters (see also \cite[Lemma 5.6.]{Zenz}).

\begin{lemma}\label{saddle-point}
Let $X,Y,V,V_1,Q>0$ and $Z:=Q+X+Y+V_1+1$, and assume that
\[Y\geq Z^{3/20}, \qquad V_1\geq V\geq\frac{QZ^{1/40}}{Y^{1/2}}.\]
Suppose that $h$ is a smooth function on $\mathbb R$ with support on an interval $J$ of length $V_1$, satisfying
\[h^{(j)}(t)\ll_j XV^{-j}\]
for all $j\in\mathbb N\cup\{0\}$. Suppose that $f$ is a smooth function on $J$ such that there exists a unique point $t_0\in J$ such that $f'(t_0)=0$, and furthermore 
\[f''(t)\gg YQ^{-2},\quad f^{(j)}(t)\ll_j YQ^{-j}\quad\text{for all }j\geq 1\,\text{and }t\in J.
\]
Then
\[\int\limits_{\mathbb R}h(t)e(f(t))\,\mathrm d t=e^{\text{sgn}(f''(t_0))\pi i/4}\frac{e(f(t_0))}{\sqrt{|f''(t_0)|}}h(t_0)+O\left(\frac{Q^{3/2}X}{Y^{3/2}}\cdot\left(V^{-2}+(Y^{2/3}/Q^2)\right)\right).\]
\end{lemma}

\section{Random mollifier}

\subsection{Random model}


\noindent In order to define a random counterpart for our mollifier we require a probabilistic random model for the Hecke eigenvalues $\lambda_f(m)$, which we now introduce following \cite{Lamzouri2019}. To motivate the model, let $G^\sharp$ denote the set of conjugacy classes of $G=\mathrm{SU}(2)$ endowed with the Sato--Tate measure. Note that if $m>1$ has the prime factorisation $m=p_1^{\alpha_1}\cdots p_\ell^{\alpha_\ell}$, then we have
\[ 
\lambda_f(m)=\prod_{j=1}^\ell\lambda_f\left(p_j^{\alpha_j}\right)=\prod_{j=1}^\ell\text{Tr}\left(\text{Sym}^{\alpha_j}(g_f(p_j))\right),\]
where 
\[ 
g_f(p):=\begin{pmatrix}
e^{i\alpha_p} & \\
& e^{i\beta_p}
\end{pmatrix}
\]
with $\alpha_p,\beta_p$ being the Satake parameters of $f$ at a prime $p$. We write $g_f^\sharp(p)$ for the conjugacy class of the matrix $g_f(p)$. Then it can be shown that for a fixed set of primes $\{p_1,...,p_\ell\}$, the $\ell$-tuples $\{g_f^\sharp(p_1),...,g_f^\sharp(p_\ell)\}_{f\in B_k}$ of conjugacy classes equidistribute inside $(G^\sharp)^\ell$ as $k\longrightarrow\infty$.

This suggests the following random model for $\lambda_f(m)$. Let $\{g_p^\sharp\}_p$ be a sequence of independent random variables with values in $G^\sharp$ that are distributed according to the Sato--Tate measure. Then we define the random variable $X$ by setting $X(1)=1$ and 
\[
X(m):=\prod_{j=1}^\ell \text{Tr}\left(\text{Sym}^{\alpha_j}(g^\sharp_{p_j})\right)\]
for $m=p_1^{\alpha_1}\cdots p_\ell^{\alpha_\ell}$.

We note that 
\begin{align}\label{random-Hecke}
X(m)X(n)=\sum_{d|(m,n)}X\left(\frac{mn}{d^2}\right)
\end{align}
and $\mathbb E(X(m))=1_{m=1}$. By linearity of the expectation these give
\begin{align}\label{correlation-expectation}
\mathbb E(X(m)X(n))=\sum_{\substack{d|(m,n)\\
mn=d^2}}1=1_{m=n}.
\end{align}
These random variables $X(m)$ provide an adequate model for the Hecke eigenvalues $\lambda_f(m)$.

\subsection{Construction of a random mollifier}

Let $X$ be the random variable as above. Then we define a random mollifier
\[ 
M_1(X;d):=(\log K)^{1/2}\prod_{0\leq j\leq J}M_{1,j}(X;d),\]
where 
\[ 
M_{1,j}(X;d):=\sum_{\substack{p|n\Rightarrow p\in I_j\\
\Omega(n)\leq 2\ell_j}}\frac{\lambda(n)\chi_d(n)\nu_2(n;\ell_j)}{2^{\Omega(n)}\sqrt n}\sum_{u|n}c_n(u)X(u). 
\]
Note that we may also write
\[ 
M_1(X;d)=(\log K)^{1/2}\sum_{n\leq K^{2\delta_0}}\frac{\lambda(n)\chi_d(n)h_1(n)}{2^{\Omega(n)}\sqrt n}\sum_{u|n}c_n(u)X(u).
\]
This is a random counterpart for $M_g(d)^2$ as can be seen by comparing the previous display to (\ref{Mollifier-power}) for $\ell=1$. 

\section{Character sum}

\noindent In the proof of Proposition \ref{fourth-moment-with-average} we require the evaluation of a certain character sum. In this section we achieve this task. In order to do this the following result is needed.

\begin{lemma}(\cite[Theorem 2.1.2.]{BEW})\label{quad_char_quad_pol} Let $p$ be an odd prime and $a,b,c$ be integers with $p\nmid a$. Then 
\[ 
\sum_{x\,(p)}\left(\frac{ax^2+bx+c}p\right)=\begin{cases}
-\left(\frac ap\right) & \text{if }b^2-4ac\not\equiv 0\,(p) \\
(p-1)\left(\frac ap\right) & \text{if }b^2-4ac\equiv 0\,(p)
\end{cases}\]
\end{lemma}
\noindent The main result of this section is the following. 

\begin{prop}\label{twisted-Kloosterman-sum}
Let $d$ be an odd fundamental discriminant and let $c,u,v,$ and $\eta$ be natural numbers.
\begin{enumerate}
\item Assume that $v\eta=u[c,d]^2/c^2$. Then the sum 
\begin{align*}
\sum_{x\,([c,d])}\sum_{w\,([c,d])}\chi_d(x)\chi_d(w)S\left(xw,u;c\right)e\left(\frac{xv+w\eta}{[c,d]}\right)
\end{align*}
vanishes unless\footnote{Note also that the condition $d|c$ forces $v\eta=u$.} $d|c$, in which case it equals 
\[ 
\frac{c\cdot\varphi(c)}{\varphi(|d|)}\chi_d(u)(-1)^{\#\{p|d\}}\prod_{\substack{p|d\\
p|\frac cd}}(1-p).
\]
Here $\#\{p|d\}$ is the number of primes dividing $d$. 

\item The sum
\[
\sum_{x\,([c,d])}\sum_{w\,([c,d])}\chi_d(x)\chi_d(w)S\left(xw,u;c\right)
\]
vanishes unless $d|c$ and $(d,c/d)=1$ in which case it equals 
\[
\text{sgn}(d)\cdot c\cdot\varphi(|d|)\chi_d(u)\Xi_{c/|d|}(u),
\]
where $\Xi_r(u)$ is the Ramanujan sum given by
\[
\Xi_r(u):=\mathop{{\sum}^{\makebox[0pt][l]{$*$}}}_{x\,(r)}e\left(\frac{ux}{r}\right).
\]
\end{enumerate}
\end{prop}

\begin{proof}
We begin by proving the part $(1)$. Suppose first that $d\nmid c$. Then there exists a prime $p|d$ for which $p\nmid c$. From the assumption $v\eta=u[c,d]^2/c^2$ it follows that $p|v\eta$. By symmetry we may assume that $p|v$.

After opening the Kloosterman sum and rearranging our task is to evaluate
\begin{align}\label{char-sum-factor}
\mathop{{\sum}^{\makebox[0pt][l]{$*$}}}_{\gamma\,(c)}\sum_{w\,([c,d])}\chi_d(w)e\left(\frac{u\overline\gamma}c+\frac{w\eta}{[c,d]}\right)\sum_{x\,([c,d])}\chi_d(x)e\left(\frac{xw\gamma}c+\frac{xv}{[c,d]}\right).
\end{align}
Let us focus on the inner sum. We compute
\begin{align*}
\sum_{x\,([c,d])}\chi_d(x)e\left(\frac{xw\gamma}c+\frac{xv}{[c,d]}\right)&=\sum_{h=0}^{[c,d]/d-1}\sum_{\ell=0}^{d-1}\chi_d(hd+\ell)e\left((hd+\ell)\left(\frac{w\gamma}c+\frac v{[c,d]}\right)\right)\\
&=\sum_{\ell=0}^{d-1}\chi_d(\ell)e\left(\ell\left(\frac{w\gamma}c+\frac v{[c,d]}\right)\right)\sum_{h=0}^{[c,d]/d-1}e\left(hd\left(\frac{w\gamma}c+\frac v{[c,d]}\right)\right).
\end{align*}
Clearly the inner $h$-sum is given by
\begin{align*}
\begin{cases}
\frac{[c,d]}d & \text{if }w\gamma\cdot\frac d{(c,d)}+v\equiv 0\,\left(\frac{[c,d]}d\right) \\
0 & \text{otherwise}
\end{cases}
\end{align*}
The $\ell$-sum is a usual Gauss sum and it evaluates to 
\[ 
\sqrt d\chi_d\left(\frac{w\gamma\cdot\frac d{(c,d)}+v}{[c,d]/d}\right)
\]
using the identity $(c,d)\cdot[c,d]=cd$.

Let us now write
\[ 
w\gamma\cdot\frac d{(c,d)}+v=h\cdot\frac{[c,d]}d
\]
for some integer $h$. By the assumption made in the beginning of the proof, both summands on the left-hand side are divisible by $p$. On the other hand, by construction $p$ does not divide the number $[c,d]/d$ as $d$ is squarefree. Thus $p|h$. But in this case 
\[ 
\chi_d\left(\frac{w\gamma\cdot\frac d{(c,d)}+v}{[c,d]/d}\right)=\chi_d(h)=0
\]
as $p|d$, and so the whole sum vanishes. 

Thus we are reduced to the case $d|c$. Note that in this case $[c,d]=c$ and $(c,d)=d$. All the computations above are still valid so taking into account the evaluation of the $h$- and $\ell$-sums it follows that (\ref{char-sum-factor}) is given by 
\[ 
\frac c{\sqrt d}\mathop{{\sum}^{\makebox[0pt][l]{$*$}}}_{\gamma\,(c)}\,\sum_{w\,(c)}\chi_d(w)e\left(\frac{v\eta\overline\gamma+w\eta}c\right)\chi_d\left(\frac{d(w\gamma+v)}c \right)\cdot 1_{w\gamma+v\equiv 0\,(c/d)}.
\]
Note that writing $w\gamma+v=h\cdot\frac cd$, the $w$-sum can be written as
\[ 
\sum_{h\,(d)}\chi_d\left(\overline\gamma\left(h\cdot\frac cd-v\right)\right)\chi_d(h)e\left(\frac{\eta h\overline\gamma}d\right),
\]
and so at this point (\ref{char-sum-factor}) is given by 
\[ 
\frac c{\sqrt d}\mathop{{\sum}^{\makebox[0pt][l]{$*$}}}_{\gamma\,(c)}\chi_d(\overline\gamma)\sum_{h\,(d)}\chi_d\left(h\left(h\cdot\frac cd-v\right)\right)e\left(\frac{\eta h\overline\gamma}d\right).
\]
To treat the subtraction inside a multiplicative character we write, using \cite[(3.12)]{Iwaniec-Kowalski2004},
\[ 
\chi_d\left(h\cdot\frac cd-v\right)=\frac1{\tau(\chi_d)}\sum_{b\,(d)}\chi_d(b)e\left(\frac{b\left(h\cdot\frac cd-v\right)}d\right).
\]  
Thus
\[ 
\sum_{h\,(d)}\chi_d\left(h\left(h\cdot\frac cd-v\right)\right)e\left(\frac{\eta h\overline\gamma}d\right)=\sum_{b\,(d)}\chi_d(b)e\left(-\frac{bv}d\right)\frac1{\tau(\chi_d)}\sum_{h\,(d)}\chi_d(h)e\left(\frac{h\left(b\cdot\frac cd+\eta\overline\gamma\right)}d\right).
\]
The inner Gauss sum equals
\[ 
\sqrt d\chi_d\left(b\cdot\frac cd+\eta\overline\gamma\right)
\]
and we also have $\tau(\chi_d)=\sqrt d$. At this point our sum is given by 
\begin{align*} 
&\frac c{\sqrt d}\mathop{{\sum}^{\makebox[0pt][l]{$*$}}}_{\gamma\,(c)}\chi_d(\overline\gamma)\sum_{b\,(d)}\chi_d\left(b^2\cdot\frac cd+b\eta\overline\gamma\right)e\left(-\frac{bv}d\right)\\
&=\frac c{\sqrt d}\sum_{b\,(d)}e\left(-\frac{bv}d\right)\mathop{{\sum}^{\makebox[0pt][l]{$*$}}}_{\gamma\,(c)}\chi_d\left(b\eta\gamma^2+b^2\cdot\frac cd\gamma\right).
\end{align*}
Now we use the complete multiplicativity of $\chi_d$ and the Chinese remainder theorem as well as apply Lemma \ref{quad_char_quad_pol} to the inner sum to see that 
\begin{align*}
\mathop{{\sum}^{\makebox[0pt][l]{$*$}}}_{\gamma\,(c)}\chi_d\left(b\eta\gamma^2+b^2\cdot\frac cd\gamma\right)&=\frac{\varphi(c)}{\varphi(|d|)}\mathop{{\sum}^{\makebox[0pt][l]{$*$}}}_{\gamma\,(d)}\chi_d\left(b\eta\gamma^2+b^2\cdot\frac cd\gamma\right)\\
&=\frac{\varphi(c)}{\varphi(|d|)}\chi_d(b\eta)(-1)^{\#\{p|d\}}\prod_{\substack{p|d\\
p|\frac cd}}(1-p). 
\end{align*} 
Thus we have shown that the sum we started with can be written as 
\[ 
\frac{c\cdot\varphi(c)}{\varphi(|d|)\cdot \sqrt d}\chi_d(\eta)(-1)^{\#\{p|d\}}\prod_{\substack{p|d\\
p|\frac cd}}(1-p)\sum_{b\,(d)}\chi_d(b)e\left(-\frac{bv}d\right)
\]
when $d|c$.

The final remaining sum is again a Gauss sum that equals $\sqrt d\chi_d(v)$. This completes the proof of $(1)$ by recalling $v\eta=u$ (for $d|c$). 

Moving to the proof of part $(2)$, we again begin by opening the Kloosterman sum. The sum in question equals
\begin{align}\label{opened-sum}
\mathop{{\sum}^{\makebox[0pt][l]{$*$}}}_{\gamma\,(c)}e\left(\frac{u\overline\gamma}c\right)\sum_{x\,([c,d])}\sum_{w\,([c,d])}\chi_d(x)\chi_d(w)e\left(\frac{\gamma xw}c\right).
\end{align}
Similarly as in the proof of part $(1)$, we see that the sum vanishes unless $d|c$. We assume this condition from now on. Next we apply the well-known identity
\begin{align}\label{char-exp-id}
\sum_{w\,(c)}\chi_d(w)e\left(\frac{bw}c\right)=\begin{cases}
\frac c{|d|}\cdot\tau(\chi_d)\chi_d(b|d|/c) & \text{if }c|bd \\
0 & \text{otherwise}
\end{cases}
\end{align}
with the choice $b=\gamma x$. We observe that the sum vanishes unless $c|\gamma xd$. But as $(\gamma,c)=1$ this can only happen when $c|xd$. In this case, if $(c/d,d)>1$ we have 
\[
\chi_d(x)=\chi_d\left(\frac cd \cdot \frac{xd}c\right)=0,
\]
and so the sum also vanishes. Thus we also assume that $(c/d,d)=1$. By using the identity (\ref{char-exp-id}) we compute
\begin{align*}
\sum_{x\,([c,d])}\sum_{w\,([c,d])}\chi_d(x)\chi_d(w)e\left(\frac{\gamma xw}c\right)&=\frac c{|d|}\cdot\tau(\chi_d)\sum_{y\,(|d|)}\chi_d\left(\frac{cy}{|d|}\right)\chi_d(\gamma y)\\
&=\frac c{|d|}\cdot\tau(\chi_d)\chi_d\left(\frac{c\gamma}{|d|}\right)\sum_{y\,(|d|)}\chi_d(y)^2\\
&=\frac c{|d|}\cdot\varphi(|d|)\tau(\chi_d)\chi_d\left(\frac{\gamma c}{|d|}\right).
\end{align*}
Substituting this into the expression (\ref{opened-sum}) gives
\[
\frac c{|d|}\cdot\varphi(|d|)\tau(\chi_d)\chi_d\left(\frac c{|d|}\right)\mathop{{\sum}^{\makebox[0pt][l]{$*$}}}_{\gamma\,(c)}\chi_d(\gamma)e\left(\frac{u\overline\gamma}c\right).
\] 
Making the change of variables $\gamma\mapsto\overline\gamma$ and applying the Chinese remainder theorem we have
\begin{align*}
\mathop{{\sum}^{\makebox[0pt][l]{$*$}}}_{\gamma\,(c)}\chi_d(\gamma)e\left(\frac{u\overline\gamma}c\right)&=\mathop{{\sum}^{\makebox[0pt][l]{$*$}}}_{\gamma\,(c)}\chi_d(\gamma)e\left(\frac{u\gamma}c\right)\\
&=\tau(\chi_d)\chi_d\left(\frac{uc}{|d|}\right)c_{c/|d|}(u),
\end{align*}
where we have also used the fact that $\chi_d(\overline\gamma)=\chi_d(\gamma)$. Thus the sum we are interested equals
\[
\frac c{|d|}\varphi(|d|)\tau(\chi_d)^2\chi_d(u)c_{c/|d|}(u)
\]
when $d|c$ and $(d,c/d)=1$. On the other hand, $\tau(\chi_d)^2=\text{sgn}(d)|d|$ and so we get the claimed expression
\[
\text{sgn}(d)\cdot c\cdot\varphi(|d|)\chi_d(u)c_{c/d}(u).
\]
This completes the proof. 
\end{proof}

\section{Proof of Proposition \ref{second-moment-with-average}}

\noindent We start with the following result concerning averages of quadratic characters. The proof is standard and similar calculations can be found e.g. in \cite{Radziwill-Soundararajan2015, Lester-Radziwill2021, Jaasaari-Lester-Saha2022, Jaasaari-Lester-Saha2024}.

\begin{lemma}\label{d-summation}
Suppose that $n$ is an odd positive integer with $n\leq X^{1-\gamma}$ for some fixed small $\gamma>0$. Let $\Phi$ denote a smooth and compactly supported function in $\R^+$. Then
\[ 
\sumflat_d \chi_d(n)\Phi\left(\frac{|d|}X\right)= \frac{4X}{\pi^2}\widehat\Phi(0)\prod_{p|n}\left(1+\frac1p\right)^{-1}\cdot 1_{n=\square}+O_{\varepsilon}\left(X^{1/2+\varepsilon}\sqrt{n}\right). \]
\end{lemma}

\begin{proof}
We pick out the property that $d$ is squarefree by the identity
\begin{equation}\label{sqf-detection}
\sum_{\substack{\alpha=1\\
(\alpha,2)=1\\
\alpha^2|d}}^\infty\mu(\alpha)=\begin{cases}
1\quad\text{if }d\,\text{ is squarefree}\\
0\quad\text{otherwise}
\end{cases}
\end{equation}
Note that the above identity holds without the condition $(\alpha,2)=1$, but this can be added as by construction $(d,2)=1$. Inserting this to the expression on the left-hand side of the statement in the lemma shows that the $d$-sum is given by
\begin{align}\label{after-Poisson}
\sum_{\substack{\alpha=1\\
(\alpha,2n)=1}}^\infty\mu(\alpha)\left(\frac{\alpha^2}{n}\right)\sum_{r\equiv\overline{\alpha^2}\,(4)}\left(\frac r{n}\right)\Phi\left(\frac{r\alpha^2}X\right).
\end{align}
Let us set $Y:=X^{1/2-\eps}/\sqrt{n}$ for an arbitrarily small fixed $\eps>0$. We split the $\alpha$-sum in (\ref{after-Poisson}) into two parts corresponding to $\alpha\leq Y$ and $\alpha>Y$. For the latter terms we estimate the $r$-sum trivially and get the upper bound 
\[ 
\ll\sum_{\alpha>Y}\frac X{\alpha^2}\ll\frac{X}Y\ll_\eps X^{1/2+\eps}\sqrt{n}.
\]
For the terms with $\alpha\leq Y$ we will evaluate the $r$-sum by applying Lemma \ref{Poisson-w-char}. The terms where $n$ is a square will contribute the main term in the zero frequency term on the dual side and the rest will give the error term.

Using Lemma \ref{Gauss-type-sum} the zero frequency contribution is given by
\begin{align}\label{MT}
\frac{X}{2}\widehat\Phi(0)\sum_{\substack{\alpha=1\\
(\alpha,2n)=1}}^\infty\frac{\mu(\alpha)}{\alpha^2}\frac{\varphi(n)}{n}\cdot 1_{n=\square},
\end{align}
where we have added back in the terms with $\alpha>Y$ at the cost of an error term of size $\ll_\eps X^{1/2+\varepsilon}\sqrt{n}$.

A simple computation shows that, recalling $n$ is odd,
\begin{align}\label{simplification}
\sum_{\substack{\alpha=1\\
(\alpha,2n)=1}}^\infty\frac{\mu(\alpha)}{\alpha^2}\cdot\frac{\varphi(n)}{n}=\prod_{p\neq 2}\left(1-\frac1{p^2}\right)\prod_{p|n}\left(1+\frac1p\right)^{-1}, 
\end{align}
leading to (\ref{MT}) being equal to 
\begin{align}\label{productform}
\frac{X}{2}\widehat\Phi(0)\prod_{p\neq 2}\left(1-\frac1{p^2}\right)\prod_{p|n}\left(1+\frac1p\right)^{-1}\cdot 1_{n=\square}.
\end{align}

It remains to bound the contribution coming from the other terms (i.e. $\ell\neq 0$) after applying Lemma \ref{Poisson-w-char}. The sum we have to estimate after an application of Poisson summation is
\[ 
\frac1{n}\sum_{\substack{\alpha\leq Y\\
(\alpha,2n)=1}}^\infty\mu(\alpha)\sum_{\ell\neq 0}\widehat \Phi\left(\frac{X\ell}{4n\alpha^2}\right)e\left(\frac{\ell\overline{\alpha^2n}}{4}\right)\tau_\ell(n).
\] 
But, as $X/n\alpha^2\geq X^{2\eps}$ in this case, the $\ell$-sum is, say, $\ll X^{-1}$ due to the rapid decay of $\widehat \Phi$ and using the trivial estimate $|\tau_\ell(n)|\leq n$. Thus the whole sum is bounded by $\ll Y/X\ll_\eps X^{1/2+\eps}\sqrt{n}$, concluding the proof.
\end{proof}

We now begin the computation of the second moment of $|c_g(|d|)|$. We will only prove the lower bound 
\[
\sum_{k\in\Z}h\left(\frac{2k-1}K\right)\sum_{g\in B^+_{k+\frac12}}\alpha_g\sumflat_d\, |c_g(|d|)|^2 M_g(d)^2\phi\left(\frac{|d|}X\right)\gg XK,
\]
since the proof of the upper bound is similar. By Waldspurger's formula and (\ref{Mollifier-power}) we have to estimate the average 
\begin{align*}
&\sum_{k\in\Z}h\left(\frac{2k-1}K\right)\sum_{f\in B_k}\omega_f\sumflat_d\, L\left(\frac12,f\otimes\chi_d\right)\phi\left(\frac{|d|}X\right)M_f(d)^2 \\
&=(\log K)^{1/2}\sum_{k\in\Z}h\left(\frac{2k-1}K\right)\sum_{f\in B_k}\omega_f\sumflat_d\, L\left(\frac12,f\otimes\chi_d\right)\phi\left(\frac{|d|}X\right) \sum_{n\leq K^{2\delta_0}}\frac{h_1(n)a_f(n)\lambda(n)\chi_d(n)}{2^{\Omega(n)}\sqrt n}.
\end{align*}
Let us first use the approximate functional equation and then execute the $f$-sum. For the latter, note that using (\ref{a-function}) and applying the Petersson formula we have
\begin{align*}
\sum_{f\in B_k}\lambda_f(m)a_f(n)&=\sum_{f\in B_k}\lambda_f(m)\sum_{u|n}c_n(u)\lambda_f(u)\\
&=\sum_{u|n}c_n(u)\cdot 1_{m=u}+\text{error},
\end{align*}
where the error term contributes
\begin{align*}
&(\log K)^{1/2}\sum_{k\in\Z}h\left(\frac{2k-1}K\right)i^{2k}\sumflat_d\,\phi\left(\frac{|d|}X\right)\sum_{m=1}^\infty\frac{\chi_d(m)}{\sqrt m}V_k\left(\frac m{|d|}\right)\\
&\qquad\qquad\times\sum_{n\leq K^{2\delta_0}}\frac{h_1(n)\lambda(n)\chi_d(n)}{2^{\Omega(n)}\sqrt n}\sum_{u|n}c_n(u)\sum_{c=1}^\infty\frac{S(m,u;c)}c J_{2k-1}\left(\frac{4\pi\sqrt{mu}}c\right)
\end{align*}
to the original sum. This gives a negligible contribution by estimating trivially using Weil's bound for Kloosterman sums and the bound (\ref{J-Bessel}) as the $m$-sum is essentially supported for $m\leq |d|K^{1+\eps}$ in which case $mu\leq mn\ll_\eps K^{1+2\delta_0+\eps}X$, and recalling that $\eta_1$ is chosen so that $\delta_0<1/2$ as well as $X\ll\sqrt K$.

The main term contributes 
\begin{align*}
(\log K)^{1/2}\sum_{k\in\Z}h\left(\frac{2k-1}K\right)\sumflat_d \phi\left(\frac{|d|}X\right)\sum_{n\leq K^{2\delta_0}}\frac{h_1(n)\lambda(n)\chi_d(n)}{2^{\Omega(n)}\sqrt n}\sum_{u|n}c_n(u)\frac{\chi_d(u)}{\sqrt u}V_k\left(\frac u{|d|}\right)
\end{align*}
to the original sum. We may replace $V_k(u/|d|)$ by $1$ using (\ref{weight-function-asymptotics}), which produces an error
\[ 
\ll_\eps\sqrt{Xu} K^{1/2+\delta_0+\eps}=o(XK).
\]
Here the final estimate holds as we can guarantee that $\delta_0$ is sufficiently small.

Let us define a multiplicative function  
\[ 
\iota(n):=\sum_{\substack{u|n\\
nu=\square}}\frac{c_n(u)}{\sqrt u}. 
\]
At this point we evaluate the $d$-sum using Lemma \ref{d-summation} and the $k$-sum by (\ref{simple-Poisson}) to see that the main term is
\begin{align}\label{second-moment-main-term}
&\asymp XK(\log K)^{1/2}\prod_{0\leq j\leq J}\sum_{\substack{p|n\Rightarrow p\in I_j\\
\Omega(n)\leq 2\ell_j}}\frac{\lambda(n)\iota(n)\nu_2(n;\ell_j)}{2^{\Omega(n)}\sqrt n}\prod_{p|\text{rad}(n)}\left(1+\frac1p\right)^{-1}.
\end{align} 
Here the implied constants depend on the weight functions $h$ and $\phi$.

Next we estimate the individual terms in the product. By Rankin's trick and the properties of $\nu_r(n;\ell)$ listed after (\ref{Mollifier-power}) we have
\begin{align*}
&\sum_{\substack{p|n\Rightarrow p\in I_j\\
\Omega(n)\leq 2\ell_j}}\frac{\lambda(n)\iota(n)\nu_2(n;\ell_j)}{2^{\Omega(n)}\sqrt n}\prod_{p|\text{rad}(n)}\left(1+\frac1p\right)^{-1}\\
&=\sum_{p|n\Rightarrow p\in I_j}\frac{\lambda(n)\iota(n)\nu_2(n)}{2^{\Omega(n)}\sqrt n}\prod_{p|\text{rad}(n)}\left(1+\frac1p\right)^{-1}+O\left(\frac1{2^{\ell_j}}\sum_{p|n\Rightarrow p\in I_j}\frac{\iota(n)\nu_2(n;\ell_j)}{\sqrt n}\prod_{p|\text{rad}(n)}\left(1+\frac1p\right)^{-1}\right).
\end{align*}
Note that the error term is 
\[ 
\ll \frac1{2^{\ell_j}}\prod_{p\in I_j}\left(1+O\left(\frac 1p\right)\right)\ll \frac{1_{j=0}\cdot(\log K)^{O(1)}+1}{2^{\ell_j}}.\]
Using multiplicativity we also have
\begin{align*}
&\sum_{p|n\Rightarrow p\in I_j}\frac{\lambda(n)\iota(n)\nu_2(n)}{2^{\Omega(n)}\sqrt n}\prod_{p|\text{rad}(n)}\left(1+\frac1p\right)^{-1}\\
&\qquad\qquad=\prod_{p\in I_j}\left(1-\frac{\iota(p)\nu_2(p)}{2\sqrt p}\left(1+\frac1p\right)^{-1}+\frac{\iota(p^2)\nu_2(p^2)}{4p}\left(1+\frac 1p\right)^{-1}+O\left(\frac1{p^{3/2}}\right)\right).
\end{align*}
Combining the previous estimates with the fact that (which follows from $\iota(p)=1/\sqrt p$)
\[
\prod_{p\in I_j}\left(1-\frac{\iota(p)\nu_2(p)}{2\sqrt p}\left(1+\frac1p\right)^{-1}+\frac{\iota(p^2)\nu_2(p^2)}{4p}\left(1+\frac 1p\right)^{-1}+O\left(\frac1{p^2}\right)\right)^{-1}\ll 1_{j=0}\cdot (\log K)^{O(1)}+1\]
it follows that the main term (\ref{second-moment-main-term}) is
\begin{align*}
&\asymp XK(\log K)^{1/2}\prod_{c<p\leq K^{\theta_J}}\left(1-\frac{\iota(p)\nu_2(p)}{2\sqrt p}\left(1+\frac1p\right)^{-1}+\frac{\iota(p^2)\nu_2(p^2)}{4p}\left(1+\frac 1p\right)^{-1}+O\left(\frac1{p^2}\right)\right)\\
&\qquad\qquad\qquad\times\prod_{0\leq j\leq J}\left(1+O\left(\frac{1_{j=0}\cdot(\log K)^{O(1)}+1}{2^{\ell_j}}\right)\right).
\end{align*}
To estimate the latter product from below note that
\[ 
\prod_{0\leq j\leq J}\left(1+O\left( \frac{1_{j=0}\cdot(\log K)^{O(1)}+1}{2^{\ell_j}}\right)\right)=1+O\left(\frac1{2^{\ell_J}}\right)\geq\frac12
\]
as $\eta_2$ was chosen to be sufficiently small. 

Moreover, by Mertens' theorem the first product is bounded from below by
\[ 
\prod_{c<p<K^{\theta_J}}\left(1-\frac1{2p}+O\left(\frac1{p^2}\right)\right)\gg (\log K)^{-1/2}
\]
as $\iota(p)=1/\sqrt p$, $\iota(p^2)=1+1/p$, and $\nu_2(p)=\nu_2(p^2)=2$.

Hence, using these we get that (\ref{second-moment-main-term}) is
\[ 
\gg XK(\log K)^{1/2}(\log K)^{-1/2}\gg XK.\]
This finishes the proof of Proposition \ref{second-moment-with-average}. \qed

\section{Relating the Fourth moment to a random model}

\noindent We now begin the estimation of the mollified fourth moment of the Fourier coefficients. Here our argument combines ingredients from \cite{Bui-Evans-Lester-Pratt2022} and \cite{Khan2010}. Indeed, in order to compute the mollified fourth moment many technical calculations simplify by relating the original sum to a random model. The starting point is a lemma, which plays a key role in the proof of Proposition \ref{fourth-moment-with-average}. Let us define a random Dirichlet series
\[ 
L(X;d,k):=2(1+\varepsilon_{k,d})\sum_{m_1=1}^\infty\sum_{m_2=1}^\infty\frac{X(m_1)X(m_2)\chi_{d}(m_1m_2)}{\sqrt{m_1m_2}}V_k\left(\frac {m_1}{|d|}\right)V_k\left(\frac {m_2}{|d|}\right).
\]
\noindent Note that this is well-defined by the bounds (\ref{der-bounds}). 

\begin{lemma}\label{twisted-second-moment-averaged}
Let $n\leq K^{4\delta_0}$ be an odd positive integer and suppose that $X\ll\sqrt K$. Then we have
\begin{align*}
&\sumflat_{d}\chi_d(n)\phi\left(\frac{|d|}X\right)\sum_{k\in\Z}h\left(\frac{2k-1}K\right)\sum_{f\in B_k}\omega_f L\left(\frac12,f\otimes\chi_d\right)^2a_f(n)\\
&=\sumflat_{d}\phi\left(\frac{|d|}X\right)\chi_d(n)\sum_{k\in\Z}h\left(\frac{2k-1}K\right)\mathbb E\left(L(X;d,k)\sum_{u|n}c_n(u)X(u)\right)\\
&\qquad+K\widehat h(0)\sumflat_{d}\,\frac{\varphi(|d|)}{|d|}\phi\left(\frac{|d|}X\right)\sum_{u|n}\frac{c_n(u)}{\sqrt u}\chi_d(nu)\frac1{2\pi i}\int\limits_{(1+\eps)}\frac{e^{2x^2}}{x^2}u^{-x}\sum_{v|u}v^{2x}\,\mathrm d x+O_\eps\left(XK^{1/2+\varepsilon}\right).
\end{align*}
\end{lemma}

\begin{proof}
Using the approximate functional equation, and noting that $(1+\varepsilon_{k,d})^2=2(1+\varepsilon_{k,d})$, we have that the sum on the left-hand side equals
\begin{align*}
&2\sumflat_{d}\,\phi\left(\frac{|d|}X\right)\sum_{k\in\Z}h\left(\frac{2k-1}K\right)(1+\varepsilon_{k,d})\\
&\qquad\qquad\times\sum_{f\in B_k}\omega_f\sum_{m_1=1}^\infty\sum_{m_2=1}^\infty \frac{\lambda_f(m_1)\lambda_f(m_2)\chi_d(m_1m_2n)}{\sqrt{m_1m_2}}V_k\left(\frac{m_1}{|d|}\right)V_k\left(\frac {m_2}{|d|}\right)\sum_{u|n}c_n(u)\lambda_f(u).  
\end{align*}
Rearranging and using the Hecke relations this can be written as
\begin{align*}
&2\sumflat_{d}\,\phi\left(\frac{|d|}X\right)\sum_{k\in\Z}h\left(\frac{2k-1}K\right)(1+\varepsilon_{k,d})\\
&\qquad\qquad\times\sum_{m_1=1}^\infty\sum_{m_2=1}^\infty \frac{\chi_d(m_1m_2n)}{\sqrt{m_1m_2}}V_k\left(\frac{m_1}{|d|}\right)V_k\left(\frac {m_2}{|d|}\right)\sum_{u|n}c_n(u)\sum_{j|(m_1,m_2)}\sum_{f\in B_k}\omega_f\lambda_f\left(\frac{m_1m_2}{j^2}\right)\lambda_f(u). 
\end{align*}
Now an application of the Petersson formula shows that this is
\begin{align*}
&2\sumflat_{d}\,\phi\left(\frac{|d|}X\right)\sum_{k\in\Z}h\left(\frac{2k-1}K\right)(1+\varepsilon_{k,d})\\
&\qquad\times\sum_{m_1=1}^\infty\sum_{m_2=1}^\infty \frac{\chi_d(m_1m_2n)}{\sqrt{m_1m_2}}V_k\left(\frac{m_1}{|d|}\right)V_k\left(\frac {m_2}{|d|}\right)\sum_{u|n}c_n(u)\sum_{j|(m_1,m_2)}1_{m_1m_2/j^2=u}+\text{off-diagonal contribution},
\end{align*}
where the off-diagonal contribution is given by the sum of two terms
\begin{align}\label{off-diagonal-sums}
&4\pi\sumflat_{d}\,\phi\left(\frac{|d|}X\right)\sum_{k\in\Z}h\left(\frac{2k-1}K\right)\sum_{m_1=1}^\infty\sum_{m_2=1}^\infty \frac{\chi_d(m_1m_2n)}{\sqrt{m_1m_2}}V_k\left(\frac{m_1}{|d|}\right)V_k\left(\frac {m_2}{|d|}\right) \nonumber\\
&\qquad\qquad\times\sum_{u|n}c_n(u)\sum_{j|(m_1,m_2)}i^{2k}\sum_{c=1}^\infty\frac{S(m_1m_2/j^2,u;c)}{c}J_{2k-1}\left(\frac{4\pi\sqrt{m_1m_2u}}{jc}\right) \nonumber\\
&+4\pi\sumflat_{d}\,\text{sgn}(d)\phi\left(\frac{|d|}X\right)\sum_{k\in\Z}h\left(\frac{2k-1}K\right)\sum_{m_1=1}^\infty\sum_{m_2=1}^\infty \frac{\chi_d(m_1m_2n)}{\sqrt{m_1m_2}}V_k\left(\frac{m_1}{|d|}\right)V_k\left(\frac {m_2}{|d|}\right)\\
&\qquad\qquad\times\sum_{u|n}c_n(u)\sum_{j|(m_1,m_2)}\sum_{c=1}^\infty\frac{S(m_1m_2/j^2,u;c)}{c}J_{2k-1}\left(\frac{4\pi\sqrt{m_1m_2u}}{jc}\right). \nonumber
\end{align}
For the main term we use the relations (\ref{random-Hecke}) and (\ref{correlation-expectation}) to write it as 
\begin{align*}
& 2\sumflat_{d}\,\phi\left(\frac{|d|}X\right)\sum_{k\in\Z}h\left(\frac{2k-1}K\right)(1+\varepsilon_{k,d})\\
&\qquad\qquad\times\sum_{m_1=1}^\infty\sum_{m_2=1}^\infty \frac{\chi_d(m_1m_2n)}{\sqrt{m_1m_2}}V_k\left(\frac{m_1}{|d|}\right)V_k\left(\frac {m_2}{|d|}\right)\sum_{u|n}c_n(u)\sum_{j|(m_1,m_2)}\mathbb E\left(X\left(\frac{m_1m_2}{j^2}\right)X(u)\right)\\
&=2\sumflat_{d}\phi\left(\frac{|d|}X\right)\sum_{k\in\Z}h\left(\frac{2k-1}K\right)(1+\varepsilon_{k,d})\\
&\qquad\qquad\qquad\times\sum_{m_1=1}^\infty\sum_{m_2=1}^\infty \frac{\chi_d(m_1m_2n)}{\sqrt{m_1m_2}}V_k\left(\frac{m_1}{|d|}\right)V_k\left(\frac {m_2}{|d|}\right)\sum_{u|n}c_n(u)\mathbb E\left(X(m_1)X(m_2)X(u)\right)\\
&=\sumflat_{d}\,\phi\left(\frac{|d|}X\right)\sum_{k\in\Z}h\left(\frac{2k-1}K\right)\mathbb E\left(L(X;d,k)\chi_d(n)\sum_{u|n}c_n(u) X(u)\right).
\end{align*}
Next we will analyse the off-diagonal contribution in more detail. Let us first treat the first term in (\ref{off-diagonal-sums}). We start by executing the $k$-sum using (\ref{J-Bessel-avg-II}) to see that the first term in the off-diagonal contribution equals
\begin{align}\label{after-k-sum}
&-\sqrt\pi K\sumflat_d\phi\left(\frac{|d|}X\right)\chi_d(n)\sum_{m_1=1}^\infty\sum_{m_2=1}^\infty\frac{\chi_d(m_1)\chi_d(m_2)}{m_1^{3/4}m_2^{3/4}}\sum_{u|n}c_n(u)u^{-1/4} \nonumber\\
&\qquad\times\Im\left(e^{-2\pi i/8}\sum_{(j,d)=1}\frac1{j}\sum_{c=1}^\infty\frac{S(m_1m_2,u;c)}{\sqrt c} e\left(\frac{2\sqrt{m_1m_2u}}{c}\right)W_K^{(2)}\left(\frac{m_1j}{|d|},\frac {m_2j}{|d|},\frac{K^2c}{8\pi\sqrt{m_1m_2u}}\right)\right)\\
&\qquad\qquad\qquad\qquad+O_\eps\left(X^{3+4\delta_0+\eps}/K^2\right), \nonumber
\end{align}
where the error comes from estimating the error term using (\ref{J-Bessel-avg-II}) and (\ref{error-bound}) for $A_1=A_2=\varepsilon$. Observe that all the sums are convergent by the bounds for the function $W_K^{(2)}$.

Splitting the $m_1$- and $m_2$-sums into congruence classes modulo $[c,d]$, these sums may be written as
\begin{align}\label{division-to-congruence-classes}
&\sum_{x\,([c,d])}\sum_{w\,([c,d])}S(xw,u;c)\chi_d(x)\chi_d(w)\\
&\qquad\qquad\times\sum_{m_1\equiv x\,([c,d])}\sum_{m_2\equiv w\,([c,d])}m_1^{-3/4}m_2^{-3/4}e\left(\frac{2\sqrt{m_1m_2u}}{c}\right)W_K^{(2)}\left(\frac{m_1j}{|d|},\frac {m_2j}{|d|},\frac{K^2c}{8\pi\sqrt{m_1m_2u}}\right)\nonumber.
\end{align}
Applying Poisson summation to the $m_1$-sum shows that it equals
\[\frac1{[c,d]}\sum_{v\in\Z} e\left(\frac{xv}{[c,d]}\right)\int\limits_{\mathbb R}y^{-3/4}e\left(\frac{2\sqrt{ym_2u}}{c}-\frac{yv}{[c,d]}\right)W_K^{(2)}\left(\frac {yj}{|d|},\frac{m_2j}{|d|},\frac{K^2c}{8\pi\sqrt{ym_2u}}\right)\,\mathrm d y.\]
Similarly, applying Poisson summation to the $m_2$-sum gives, after some computations, that our double sum equals
\begin{align*}
&\sum_{x\,([c,d])}\sum_{w\,([c,d])}S\left(xw,u;c\right)\chi_d(x)\chi_d(w)\frac1{[c,d]^2}\sum_{v\in\Z}\sum_{\eta\in\Z} e\left(\frac{xv+w\eta}{[c,d]}\right)\cdot\mathcal I,
\end{align*}
where 
\begin{align*}
\mathcal I:=\int\limits_{\mathbb R}\int\limits_{\mathbb R}y^{-3/4}z^{-3/4}e\left(\frac{2\sqrt{yzu}}{c}-\frac{yv}{[c,d]}-\frac{z\eta}{[c,d]}\right)W_K^{(2)}\left(\frac {yj}{|d|},\frac{zj}{|d|},\frac{K^2c}{8\pi\sqrt{yzu}}\right)\,\mathrm d y\,\mathrm d z.
\end{align*}

We first focus on the $y$-integral. Let us set
\[ 
f(y):=\frac{2\sqrt{yzu}}{c}-\frac{yv}{[c,d]}-\frac{z\eta}{[c,d]}.\]
Note that for $v\leq 0$, the derivative $f'(y)$ is bounded away from zero on the support, and repeated integration by parts gives a negligible contribution. Hence we may assume $v\geq 1$ before applying Lemma \ref{saddle-point}. An easy computation shows that the integral has a saddle point at
\[y_0:=\frac{[c,d]^2zu}{c^2v^2}.\]
Similarly a straightforward computation gives 
\[f(y_0)=z\left(\frac{u[c,d]}{c^2v}-\frac{\eta}{[c,d]}\right)\qquad\text{and}\qquad f''(y_0)=-\frac12\cdot\frac{c^2v^3}{zu[c,d]^3}.\]
Note that by (\ref{W_K-2-bound}) the function $W_K^{(2)}$ decays rapidly unless 
\begin{align}\label{restrictions}
|y|\ll_\eps |d|K^{1+\eps}/j,\qquad |z|\ll_\eps |d|K^{1+\eps}/j, \quad\text{and}\quad |yz|\gg_\eps K^{4-\eps}c^2/u.
\end{align}
Under these restrictions we have 
\[ 
c^2j^2\ll_\eps \frac{X^{2+\eps}u}{K^2}.
\]
As $u\leq n\leq K^{4\delta_0}$ we deduce that $c$- and $j$-sums are both essentially supported for $c,j\ll K^{2\delta_0-1+\eps}X$. We also observe that from the restrictions (\ref{restrictions}) it follows that the stationary point only occurs when 
\[
v,\eta\ll K^{-1+\eps}X^2 u \ll K^\eps u \ll K^{4\delta_0+\eps}.   
\]
We truncate the $y$-integral smoothly to $|y|\leq K^{1+\eps}|d|/j$ with a negligible error and then apply Lemma \ref{saddle-point} with the choices $Q=V=V_1=K|d|/j$, $X=K^\eps$, and $Y=\sqrt{zuK|d|}/c\sqrt j$. It is straightforward to check that the conditions of the lemma are met under (\ref{restrictions}). We conclude that 
\[ 
\mathcal I\sim\sqrt 2e^{-\pi i/4}\sqrt c u^{-1/4}\int\limits_{\R}z^{-1}e\left( z\left(\frac{u[c,d]}{c^2v}-\frac\eta{[c,d]}\right)\right)W_K^{(2)}\left(\frac{[c,d]^2zuj}{c^2v^2|d|},\frac{zj}{|d|},\frac{K^2c^2v}{8\pi zu[c,d]}\right)\,\mathrm d z
\]
with an error $\ll K^{-5/4+\eps}X^{-5/4}j^{5/4}u^{-3/4}$, and so it follows that the expression (\ref{after-k-sum}) is, after some simplification, equal to
\begin{align}\label{off-diagonal-intermediate}
&-\sqrt{2\pi}K\sum_{u|n}\frac{c_n(u)}{\sqrt u}\sumflat_d\,\phi\left(\frac{|d|}X\right)\chi_d(n)\sum_{(j,d)=1}\frac1{j}\Im\bigg(e^{-\pi i/2}\sum_{c=1}^\infty\frac1{[c,d]^2} \sum_{v\in\Z}\sum_{\eta\in\Z}\sum_{x\,([c,d])}\sum_{w\,([c,d])}\chi_d(x)\chi_d(w)\nonumber\\
& \times S(xw,u;c)e\left(\frac{xv+w\eta}{[c,d]}\right)\int\limits_{\mathbb R}z^{-1}e\left(z\left(\frac{u[c,d]}{c^2v}-\frac\eta{[c,d]}\right)\right)W_K^{(2)}\left(\frac{zu[c,d]^2j}{v^2c^2|d|},\frac{zj}{|d|},\frac{K^2c^2v}{8\pi zu[c,d]}\right)\,\mathrm d z\bigg)\\
&\qquad\qquad\qquad\qquad+O_\eps\left(X^{2}K^{-5/2+25\delta_0/2+\eps}\right) \nonumber.
\end{align} 
Here the contribution from the error term in the asymptotics for the integral $\mathcal I$ is estimated as
\begin{align*}
& \ll_\eps KX\cdot K^{-5/4+\eps} X^{-5/4}\sum_{c\ll K^{2\delta_0-1+\eps}}\sum_{j\ll K^{2\delta_0-1+\eps}}\sum_{v\ll K^{4\delta_0+\eps}}\sum_{\eta\ll K^{4\delta_0+\eps}}j^{1/4}\\
& \ll_\eps X^{2}K^{-5/2+25\delta_0/2+\eps}
\end{align*}
Note also that the exponential phase in the integral (\ref{off-diagonal-intermediate}) vanishes when $\eta v=u[c,d]^2/c^2$. If this is not the case, we may bound the integral using the first derivative test \cite[Section 1.5.]{Huxley1996}. Indeed, note that in this case the absolute value of the derivative of the phase function is $\geq 1/c^2v[c,d]$. Using the first derivative test and estimating everything else trivially using the Weil bound for Kloosterman sums shows that (\ref{off-diagonal-intermediate}) is $\ll_\eps X^{1+\eps}/K$ when $\eta v\neq u[c,d]^2/c^2$. 

When $\eta v=u[c,d]^2/c^2$ we use Proposition \ref{twisted-Kloosterman-sum} to deduce that the main part of the first summand in the off-diagonal is
\begin{align}\label{simplified-off-diagonal}
&-\sqrt{2\pi} K\sumflat_d\,\phi\left(\frac{|d|}X\right)(-1)^{\#\{p|d\}}\sum_{u|n}c_n(u)\frac{\chi_d(nu)}{\sqrt u} \nonumber \\
&\qquad\times\sum_{v|u}\Im\left(e^{-\pi i/2}\sum_{(j,d)=1}\frac 1j\sum_{\substack{c=1\\
d|c}}^\infty \frac{\varphi(c)}{c\cdot\varphi(|d|)}\prod_{\substack{p|d\\
p|\frac cd}}(1-p)\int\limits_{\mathbb R} z^{-1}W_K^{(2)}\left(\frac{zuj}{|d|v^2},\frac{zj}{|d|},\frac{K^2cv}{8\pi zu}\right)\,\mathrm d z\right).
\end{align}
We extend the definition of $W(zuj/K|d|v^2,zj/K|d|,K^2cv/8\pi zu)$ to all real numbers $z$ by setting\\ $W(zuj/K|d|v^2,zj/K|d|,K^2cv/8\pi zu)=0$ for $z\leq 0$. Note that the resulting function is still smooth. Using the approximation (\ref{new-weight-function}) and Poisson summation we see that the integral in (\ref{simplified-off-diagonal}) equals, up to an error $\ll K^{-1}$, 
\[\int\limits_{\mathbb R} z^{-1}W\left(\frac{zuj}{K|d|v^2},\frac {zj}{K|d|},\frac{K^2cv}{8\pi zu}\right)\,\mathrm d z=\sum_{h=1}^\infty h^{-1}W\left(\frac{huj}{K|d|v^2},\frac {hj}{K|d|},\frac{K^2vc}{8\pi hu}\right)+O\left(K^{-100}\right).
\]

By Mellin inversion we have, for any $\eps>0$,
\begin{align*}
&W\left(\frac{huj}{K|d|v^2},\frac {hj}{K|d|},\frac{K^2vc}{8\pi hu}\right)\\
&=\frac1{(2\pi i)^2}\int\limits_{(1+\eps)}\int\limits_{(1+\eps)}(2\pi)^{-x-y}e^{x^2+y^2}\left(\frac{huj}{K|d|v^2}\right)^{-x}\left(\frac {hj}{K|d|}\right)^{-y}\hbar_{x+y}\left(\frac{K^2vc}{8\pi hu}\right)\,\frac{\mathrm d x\,\mathrm d y}{xy}\\
&=\frac1{(2\pi i)^3}\int\limits_{(1+\eps)}\int\limits_{(1+\eps)}\int\limits_{(1+\varepsilon)}(2\pi)^{-x-y}e^{x^2+y^2}\left(\frac{huj}{K|d|v^2}\right)^{-x}\left(\frac {hj}{K|d|}\right)^{-y}\left(\frac{8\pi hu}{K^2vc}\right)^z\widetilde\hbar_{x+y}(z)\,\mathrm d z\,\frac{\mathrm d x\,\mathrm d y}{xy},
\end{align*}
where we have used (\ref{Mellin-for-h}) in the final step. 

One also easily computes, for $\Re(z)>1$,
\begin{align*}
\sum_{\substack{c=1\\
d|c}}^\infty \frac{\varphi(c)}{c^{1+z}\cdot\varphi(|d|)}\prod_{\substack{p|d\\
p|\frac cd}}(1-p)&=\sum_{c=1}^\infty\frac{\varphi(c)}{c^{1+z}}\cdot\frac1{|d|^{1+z}}\prod_{\substack{p|d\\
p|c}}(-p) \\
&=\frac1{|d|^{1+z}}\prod_p\left(1+\sum_{j=1}^\infty\frac{\varphi(p^j)}{p^{j(1+z)}}\prod_{\substack{q|d\\
q|p^j}}(-q)\right)\\
&=\frac1{|d|^{1+z}}\prod_{p|d}\left(1+\sum_{j=1}^\infty\frac{\varphi(p^j)}{p^{j(1+z)}}(-p)\right)\prod_{p\nmid d}\left(1+\sum_{j=1}^\infty\frac{\varphi(p^j)}{p^{j(1+z)}}\right)\\
&=\frac1{|d|^{1+z}}\prod_{p|d}\frac{1+\sum_{j=1}^\infty\frac{\varphi(p^j)}{p^{j(1+z)}}(-p)}{1+\sum_{j=1}^\infty\frac{\varphi(p^j)}{p^{j(1+z)}}}\prod_p\left(1+\sum_{j=1}^\infty\frac{\varphi(p^j)}{p^{j(1+z)}}\right)\\
&=\frac{\zeta(z)}{\zeta(1+z)}\cdot |d|^{-z}\prod_{p|d}\left(\frac{p^{z}-p}{p(p^{z}-1)+p-1}\right),
\end{align*}
where in the last step we have used the fact that $d$ is squarefree. 

Substituting these into (\ref{simplified-off-diagonal}) leads us to consider, up to an error term $\ll XK^\eps$, the expression
\begin{align}\label{zeta-int}
&\sqrt{2\pi}K\sumflat_d\,\phi\left(\frac{|d|}X\right)(-1)^{\#\{p|d\}}\sum_{u|n}\frac{c_n(u)}{\sqrt u}\chi_d(nu) \nonumber \\
& \qquad\times\Im\Bigg(\frac1{(2\pi i)^3}\int\limits_{(1+\eps)}\int\limits_{(1+\eps)}\int\limits_{(1+\varepsilon)}(2\pi)^{-x-y}e^{x^2+y^2}u^{z-x}(8\pi)^zK^{x+y-2z}|d|^{x+y-z}\sum_{v|u}v^{2x-z}\widetilde\hbar_{x+y}(z) \nonumber\\
&\qquad\qquad\times\left(\sum_{(j,d)=1}\frac1{j^{1+x+y}}\right)\left(\sum_{h=1}^\infty\frac1{h^{1+x+y-z}}\right)\left(\sum_{c=1}^\infty\frac{\varphi(c)}{c^{1+z}}\cdot\frac1d\prod_{\substack{p|d\\
p|c}}(-p)\right)\,\mathrm d z\frac{\mathrm d x\,\mathrm d y}{xy}\Bigg) \nonumber\\
&=\sqrt{2\pi}K\sumflat_d\,\phi\left(\frac{|d|}X\right)(-1)^{\#\{p|d\}}\sum_{u|n}\frac{c_n(u)}{\sqrt u}\chi_d(nu)\nonumber \\
&\qquad\times\Im\Bigg(\frac1{(2\pi i)^3}\int\limits_{(1+\eps)}\int\limits_{(1+\eps)}\int\limits_{(1+\varepsilon)} (2\pi)^{-x-y}e^{x^2+y^2}u^{z-x}(8\pi)^zK^{x+y-2z}|d|^{x+y-z}\sum_{v|u}v^{2x-z}\widetilde\hbar_{x+y}(z) \nonumber \\
&\qquad\times\zeta(1+x+y)\frac{\zeta(z)}{\zeta(1+z)}\prod_{p|d}\left(\left(1-p^{-1-x-y}\right)\frac{p^{z}-p}{p(p^{z}-1)+p-1}\right)\zeta(1+x+y-z)\,\mathrm d z\,\frac{\mathrm d x\,\mathrm d y}{xy}\Bigg).
\end{align}

\noindent But this expression is clearly $\ll XK^{1/2}$ by shifting the $z$-integral, say, to the line $\Re(z)=2$ (recall Lemma \ref{Mellin-transform-lemma} and remarks after it), concluding the treatment of the first term in (\ref{off-diagonal-sums}). 

Finally, let us focus on the second term in (\ref{off-diagonal-sums}). By re-parametrising it is given by 
\begin{align*}
&4\pi \sumflat_d\,\phi\left(\frac{|d|}X\right)\text{sgn}(d)\chi_d(n)\sum_{u|n}c_n(u)\sum_{\substack{j=1\\(j,d)=1}}^\infty\frac1j\sum_{c=1}^\infty\frac1c\sum_{m_1=1}^\infty\sum_{m_2=1}^\infty\frac{\chi_d(m_1m_2)S(m_1m_2,u;c)}{\sqrt{m_1m_2}}\\
&\qquad\qquad\qquad\times\sum_{k\in\Z}h\left(\frac{2k-1}K\right)V_k\left(\frac{jm_1}{|d|}\right)V_k\left(\frac{jm_2}{|d|}\right)J_{2k-1}\left(\frac{4\pi\sqrt{m_1m_2u}}{c}\right).
\end{align*}

Using (\ref{J-Bessel-avg-III}) the $k$-sum is given by
\[
h\left(\frac{4\pi\sqrt{m_1m_2u}}{cK}\right)V_{\frac12+\frac{2\pi\sqrt{m_1m_2u}}{c}}\left(\frac{jm_1}{|d|}\right)V_{\frac12+\frac{2\pi\sqrt{m_1m_2u}}{c}}\left(\frac{jm_2}{|d|}\right)
\]
up to an error term 
\begin{align*}
&\ll_\eps XK^{-2+\eps}\sum_{j=1}^\infty\frac1j\sum_{m\ll XK^{q+\eps/j}}\sum_{n\ll XK^{q+\eps/j}}\frac1{\sqrt{mn}}\sum_{c\ll \sqrt{mnu}/K}c^{-1/2+\eps}\sum_{u|n}\sqrt u|c_n(u)|\\
&\ll_\eps \frac{X^{3/2}}{K^{1-\eps}}
\end{align*}
using the Weil bound. 

Let us concentrate on the main term. By splitting $m_1$- and $m_2$-sums into congruence classes modulo $[c,d]$ the sum we are interested in is given by 
\begin{align*}
&4\pi \sumflat_d\,\phi\left(\frac{|d|}X\right)\text{sgn}(d)\chi_d(n)\sum_{u|n}c_n(u)\sum_{\substack{j=1\\(j,d)=1}}^\infty\frac1j\sum_{c=1}^\infty\frac1c\sum_{x\,([c,d])}\sum_{w\,([c,d])}\chi_d(xw)S(xw,u;c)\\
&\times\sum_{m_1\equiv x\,([c,d])}\sum_{m_2\equiv w\,([c,d])}\frac1{\sqrt{m_1m_2}}h\left(\frac{4\pi\sqrt{m_1m_2u}}{cK}\right)V_{\frac12+\frac{2\pi\sqrt{m_1m_2u}}{c}}\left(\frac{jm_1}{|d|}\right)V_{\frac12+\frac{2\pi\sqrt{m_1m_2u}}{c}}\left(\frac{jm_2}{|d|}\right).
\end{align*}
Applying Poisson summation to both $m_1$- and $m_2$-sums gives
\begin{align*}
&\sum_{m_1\equiv x\,([c,d])}\sum_{m_2\equiv w\,([c,d])}\frac1{\sqrt{m_1m_2}}h\left(\frac{4\pi\sqrt{m_1m_2u}}{cK}\right)V_{\frac12+\frac{2\pi\sqrt{m_1m_2u}}{c}}\left(\frac{jm_1}{|d|}\right)V_{\frac12+\frac{2\pi\sqrt{m_1m_2u}}{c}}\left(\frac{jm_2}{|d|}\right)\\
&=\frac1{[c,d]^2}\sum_{\eta\in\Z}\sum_{v\in\Z} e\left(\frac{xv+w\eta}{[c,d]}\right)\\
&\qquad\qquad\qquad\times\int\limits_0^\infty\int\limits_0^\infty h\left(\frac{4\pi\sqrt{s_1s_2u}}{cK}\right)V_{\frac12+\frac{2\pi\sqrt{s_1s_2u}}{c}}\left(\frac{js_1}{|d|}\right)V_{\frac12+\frac{2\pi\sqrt{s_1s_2u}}{c}}\left(\frac{js_2}{|d|}\right)e\left(\frac{-(vs_1+\eta s_2)}{[c,d]}\right)\frac{\mathrm d s_1 \mathrm d s_2}{\sqrt{s_1s_2}}.
\end{align*}
As the function $h$ is compactly supported, we have $s_1s_2\asymp K^2c^2/u$. Using this and (\ref{der-bounds}) we see by repeated integration by parts that the double integral is negligible unless $v=\eta=0$ as $X\ll\sqrt K$. Thus, up to a negligible error, the sum we are interested in equals
\begin{align*}
&4\pi \sumflat_d\,\phi\left(\frac{|d|}X\right)\text{sgn}(d)\chi_d(n)\sum_{u|n}c_n(u)\sum_{\substack{j=1\\(j,d)=1}}^\infty\frac1j\sum_{c=1}^\infty\frac1c\sum_{x\,([c,d])}\sum_{w\,([c,d])}\chi_d(xw)S(xw,u;c)\\
&\qquad\qquad\times\int\limits_0^\infty\int\limits_0^\infty h\left(\frac{4\pi\sqrt{s_1s_2u}}{cK}\right)V_{\frac12+\frac{2\pi\sqrt{s_1s_2u}}{c}}\left(\frac{js_1}{|d|}\right)V_{\frac12+\frac{2\pi\sqrt{s_1s_2u}}{c}}\left(\frac{js_2}{|d|}\right)\frac{\mathrm d s_1 \mathrm d s_2}{\sqrt{s_1s_2}}.
\end{align*}
Using the part $(2)$ of Proposition \ref{twisted-Kloosterman-sum} and re-parametrising $r=c/|d|$ this equals
\begin{align}\label{intermediate-term-2}
&2\pi \sumflat_d\,\phi\left(\frac{|d|}X\right)\varphi(|d|)\sum_{u|n}c_n(u)\chi_d(nu)\sum_{\substack{j=1\\(j,d)=1}}^\infty\frac1j\sum_{\substack{r=1\\
(r,d)=1}}^\infty\frac{\Xi_r(u)}{|d|^2r^2}\\
&\qquad\qquad\qquad\times\int\limits_0^\infty\int\limits_0^\infty h\left(\frac{4\pi\sqrt{s_1s_2u}}{|d|rK}\right)V_{\frac12+\frac{2\pi\sqrt{s_1s_2u}}{|d|r}}\left(\frac{js_1}{|d|}\right)V_{\frac12+\frac{2\pi\sqrt{s_1s_2u}}{|d|r}}\left(\frac{js_2}{|d|}\right)\frac{\mathrm d s_1 \mathrm d s_2}{\sqrt{s_1s_2}}. \nonumber
\end{align}
We make the change of variables
\begin{align*}
\xi_1=\frac{4\pi\sqrt{s_1s_2u}}{|d|rK} \quad \text{and} \quad \xi_2=\sqrt{\frac{s_1}{s_2}}
\end{align*}
so that 
\begin{align*}
s_1=\frac{|d|rK\xi_1\xi_2}{4\pi\sqrt u} \quad \text{and} \quad s_2=\frac{|d|rK\xi_1}{4\pi\sqrt u \xi_2}.
\end{align*}
Calculating the Jacobian we see that 
\[
\frac{\mathrm d s_1\,\mathrm d s_2}{\sqrt{s_1s_2}}=\frac{|d|rK}{2\pi \sqrt u}\mathrm d \xi_1\frac{\mathrm d \xi_2}{\xi_2}
\]
and thus (\ref{intermediate-term-2}) can be written in the form
\begin{align}\label{intermediate-form}
&K\sumflat_d\,\phi\left(\frac{|d|}X\right)\frac{\varphi(|d|)}{|d|}\sum_{u|n}\frac{c_n(u)}{\sqrt u}\chi_d(nu)\sum_{\substack{j=1\\(j,d)=1}}^\infty\frac1j\sum_{\substack{r=1\\
(r,d)=1}}^\infty\frac{\Xi_r(u)}{r}\\
&\qquad\qquad\qquad\times\int\limits_0^\infty\int\limits_0^\infty h\left(\xi_1\right)V_{\frac{K\xi_1+1}2}\left(\frac{jrK\xi_1\xi_2}{4\pi\sqrt u}\right)V_{\frac{K\xi_1+1}2}\left(\frac{jrK\xi_1}{4\pi\sqrt u\xi_2}\right)\mathrm d \xi_1\frac{\mathrm d \xi_2}{\xi_2}. \nonumber
\end{align}
By applying Stirling's formula in the form (see (\ref{Stirling}))
\[
\frac{\Gamma\left(\frac{K\xi+1}2+s\right)}{\Gamma\left(\frac{K\xi+1}2\right)}=\left(\frac{K\xi}2\right)^s\left(1+O\left(\frac{|s|^2}K\right)\right)
\]
we get that 
\[
V_{\frac{K\xi_1+1}2}\left(\frac{jrK\xi_1\xi_2}{4\pi\sqrt u}\right)=\frac1{2\pi i}\int\limits_{(2)}e^{s^2}\left(\frac{jr\xi_2}{\sqrt u}\right)^{-s}\frac{\mathrm d s}s+O_B\left(K^{-1}\left(1+\frac{jr\xi_2}{\sqrt u}\right)^{-B} \right)
\]
for any $B\geq 1$, and similarly
\[
V_{\frac{K\xi_1+1}2}\left(\frac{jrK\xi_1}{4\pi\sqrt u\xi_2}\right)=\frac1{2\pi i}\int\limits_{(2)}e^{s^2}\left(\frac{jr}{\sqrt u \xi_2}\right)^{-s}\frac{\mathrm d s}s+O_B\left(K^{-1}\left(1+\frac{jr}{\sqrt u \xi_2}\right)^{-B} \right).
\]
These lead to the estimate
\begin{align*}
&\int\limits_0^\infty\int\limits_0^\infty h\left(\xi_1\right)V_{\frac{K\xi_1+1}2}\left(\frac{jrK\xi_1\xi_2}{4\pi\sqrt u}\right)V_{\frac{K\xi_1+1}2}\left(\frac{jrK\xi_1}{4\pi\sqrt u\xi_2}\right)\mathrm d \xi_1\frac{\mathrm d \xi_2}{\xi_2}\\
&\qquad\qquad\qquad=\widehat h(0)\int\limits_0^\infty W\left(\frac{jr\xi_2}{\sqrt u}\right)W\left(\frac{jr}{\sqrt u \xi_2}\right)\frac{\mathrm d \xi_2}{\xi_2}+O_B\left(K^{-1}(jr)^{-B}\right)
\end{align*}
for any $B\geq 1$, where we have defined 
\[
W(z):=\frac1{2\pi i}\int\limits_{(2)}e^{s^2}z^{-s}\frac{\mathrm d s}{s}. 
\]
Next we apply Parseval's theorem for the Mellin transform in the form
\[
\int\limits_0^\infty f(t)g(t)\,\mathrm d t=\frac1{2\pi i}\int\limits_{(2)}\widetilde f(s)\widetilde g(1-s)\,\mathrm d s. 
\]
This gives 
\[
\int\limits_0^\infty W\left(\frac{jr\xi_2}{\sqrt u}\right)W\left(\frac{jr}{\sqrt u \xi_2}\right)\frac{\mathrm d \xi_2}{\xi_2}=\frac1{2\pi i}\int\limits_{(2)}\left(\frac{j^2r^2}u\right)^{-s}e^{2s^2}\frac{\mathrm d s}{s^2}
\]
as the Mellin transform of the function $\xi\mapsto W(jr\xi/\sqrt u)$ is $s\mapsto (jr/\sqrt u)^{-s}e^{s^2}/s$ and the Mellin transform of the function $\xi\mapsto \xi^{-1}W(jr/\xi\sqrt u)$ is $s\mapsto (jr/\sqrt u)^{s-1}e^{(1-s)^2}/(1-s)$.

Substituting these into (\ref{intermediate-form}) gives the main contribution
\begin{align}\label{intermediate-form2}
K\widehat h(0)\frac{\varphi(|d|)}{|d|}\sum_{u|n}\frac{c_n(u)}{\sqrt u}\cdot\chi_d(nu)\cdot\frac1{2\pi i}\int\limits_{(2)}e^{2s^2}u^s\sum_{\substack{j=1\\
(j,d)=1}}^\infty\frac1{j^{1+2s}}\sum_{\substack{r=1\\
(r,d)=1}}^\infty\frac{\Xi_r(u)}{r^{1+2s}}\frac{\mathrm d s}{s^2}.
\end{align}
It is well-known that
\[
\sum_{r=1}^\infty\frac{\Xi_r(u)}{r^{1+2s}}=\frac{\sigma_{2s}(u)}{u^{2s}\zeta(1+2s)}
\]
for $\Re(s)>0$, which gives
\[
\sum_{\substack{r=1\\
(r,d)=1}}^\infty\frac{\Xi_r(u)}{r^{1+2s}}=\frac{\sigma_{2s}(u)}{u^{2s}\zeta(1+2s)\prod_{p|d}\left(1-p^{-1-2s}\right)}.
\]
Thus
\[
\sum_{\substack{j=1\\
(j,d)=1}}^\infty\frac1{j^{1+2s}}\sum_{\substack{r=1\\
(r,d)=1}}^\infty\frac{\Xi_r(u)}{r^{1+2s}}=\frac{\sigma_{2s}(u)}{u^{2s}}. 
\]
Note also that
\[
u^{-2s}\sigma_{2s}(u)=u^{-2s}\sum_{v|u}v^{2s}.
\]
Substituting these into (\ref{intermediate-form2}) gives
\[
K\widehat h(0)\frac{\varphi(|d|)}{|d|}\sum_{u|n}\frac{c_n(u)}{\sqrt u}\chi_d(nu)\cdot\frac1{2\pi i}\int\limits_{(2)}e^{2s^2}u^{-s}\sum_{v|u}v^{2s}\frac{\mathrm d s}{s^2},
\]
which is the desired secondary main term. This finishes the proof.  
\end{proof}
Let us define another random mollifier by 
\begin{align*}
M_2(X;d)&:=(\log K)\sum_{n\leq K^{4\delta_0}}\frac{h_2(n)\lambda(n)\chi_d(n)}{2^{\Omega(n)}\sqrt n}\sum_{u|n}c_n(u)X(u) \\
&=(\log K)\prod_{j=0}^J \underbrace{\sum_{\substack{p|n\Rightarrow p\in I_j\\
\Omega(n)\leq 4\ell_j }}\frac{\lambda(n)\chi_d(n)\nu_4(n;\ell_j)}{2^{\Omega(n)}\sqrt n}\sum_{u|n}c_n(u)X(u)}_{=: M_{2,j}(X;d)}. 
\end{align*}
This is a random counterpart for $M_g(d)^4$ (compare to (\ref{Mollifier-power}) with $\ell=2$).

An immediate consequence of the previous lemma is the evaluation of the mollified second moment.
\begin{corollary}\label{mollified-moment-connection}
Suppose that $X\ll\sqrt K$. Then we have
\begin{align}\label{off-diag-main}
&\sumflat_{d}\,\phi\left(\frac{|d|}X\right)\sum_{k\in\Z}h\left(\frac{2k-1}K\right)\sum_{f\in B_k}\omega_f L\left(\frac12,f\otimes\chi_d\right)^2M_f(d)^4 \nonumber\\
&=\sumflat_{d}\,\phi\left(\frac{|d|}X\right)\sum_{k\in\Z}h\left(\frac{2k-1}K\right)\mathbb E\left(L(X;d,k)M_2(X;d)\right) \nonumber\\
&+K(\log K)\widehat h(0)\sum_{n\leq K^{4\delta_0}}\frac{h_2(n)\lambda(n)}{2^{\Omega(n)}\sqrt n}\sumflat_{d}\,\frac{\varphi(|d|)}{|d|}\phi\left(\frac{|d|}X\right)\sum_{u|n}\frac{c_n(u)}{\sqrt u}\chi_d(nu)\frac1{2\pi i}\int\limits_{(1+\eps)}\frac{e^{2x^2}}{x^2}u^{-x}\sum_{v|u}v^{2x}\,\mathrm d x \\&\qquad\qquad\qquad\qquad\qquad\qquad+O_\eps\left(X^{1+2\delta_0+\eps}K^{1/2}\right). \nonumber
\end{align}
\end{corollary}

\section{A random computation}

\noindent Let us now define for $\min\{\Re(s_1),\Re(s_2)\}>1$,
\[
L(s_1,s_2,X;d)=\sum_{m_1=1}^\infty\sum_{m_2=1}^\infty\frac{X(m_1)X(m_2)\chi_d(m_1m_2)}{m_1^{s_1}m_2^{s_2}}=\prod_p\underbrace{\sum_{j_1=0}^\infty\sum_{j_2=0}^\infty\frac{X(p^{j_1})X(p^{j_2})\chi_d(p^{j_1+j_2})}{p^{j_1s_1+j_2s_2}}}_{=:L_p(s_1,s_2,X;d)}.
\] 
Our next goal is to estimate the contribution of primes of different sizes.

\subsection{Contribution of the primes with $K^{\theta_0}<p\leq K^{\theta_J}$}

Here we wish to estimate the expectation
\[ 
\mathbb E\left(\prod_{K^{\theta_0}<p<K^{\theta_J}}L_p\left(s_1+\frac12,s_2+\frac12;X,d\right)\prod_{j=1}^J M_{2,j}(X;d)\right).\]
This is needed for shifting contours in the proof of Lemma \ref{random-averaged}. Towards estimating this, we consider for every $1\leq j\leq J$ the term 
\begin{align}\label{exp-1}
\mathbb E\left(M_{2,j}(X;d)\prod_{p\in I_j}L_p\left(s_1+\frac12,s_2+\frac12;X,d\right)\right).
\end{align}
Let us start by expanding
\begin{align*}
&M_{2,j}(X;d)\prod_{p\in I_j}L_p\left(s_1+\frac12,s_2+\frac12;X,d\right)\\
&=\sum_{p|m_1m_2\Rightarrow p\in I_j}\frac{\chi_d(m_1m_2)}{m_1^{s_1+1/2}m_2^{s_2+1/2}}X(m_1)X(m_2)\sum_{\substack{p|n\Rightarrow p\in I_j\\
\Omega(n)\leq 4\ell_j}}\frac{\lambda(n)\chi_d(n)\nu_4(n;\ell_j)}{2^{\Omega(n)}\sqrt n}\sum_{u|n}c_n(u)X(u). 
\end{align*}
By taking the expectation and using the Hecke relations for $X(m)$ we see that (\ref{exp-1}) equals 
\begin{align}\label{after-exp}
&\sum_{p|m_1m_2\Rightarrow p\in I_j}\frac{\chi_d(m_1m_2)}{m_1^{s_1+1/2}m_2^{s_2+1/2}}\sum_{\substack{p|n\Rightarrow p\in I_j\\
\Omega(n)\leq 4\ell_j}}\frac{\lambda(n)\chi_d(n)\nu_4(n;\ell_j)}{2^{\Omega(n)}\sqrt n}\sum_{u|n}c_n(u)\sum_{h|(m_1,m_2)}1_{m_1m_2/h^2=u}\nonumber\\
&=\sum_{\substack{(h,d)=1\\
p|h\Rightarrow p\in I_j}}\frac1{h^{1+s_1+s_2}}\sum_{p|m_1m_2\Rightarrow p\in I_j}\frac{\chi_d(m_1m_2)}{m_1^{s_1+1/2}m_2^{s_2+1/2}}\sum_{\substack{p|n\Rightarrow p\in I_j\\
\Omega(n)\leq 4\ell_j\\
m_1m_2|n}}\frac{\lambda(n)\chi_d(n)\nu_4(n;\ell_j)}{2^{\Omega(n)}\sqrt n}c_n(m_1m_2).
\end{align}
For any $r>0$ and $n,\ell\in\mathbb N$,
\[ 
1_{\Omega(n)=\ell}=\frac1{2\pi i}\int\limits_{|z|=r}z^{\Omega(n)-\ell}\frac{\mathrm d z}z,\]
so that for $r\neq 1$ we have
\[ 
1_{\Omega(n)\leq\ell_j}=\frac1{2\pi i}\int\limits_{|z|=r}z^{\Omega(n)}\frac{1-z^{-\ell_j-1}}{1-z^{-1}}\frac{\mathrm d z}z.\]
Therefore, for $1<r\leq 2$ the right-hand side of (\ref{after-exp}) is given by 
\begin{align}\label{Integral-over-circle}
\frac1{2\pi i}\int\limits_{|z|=r}\frac{1-z^{-4\ell_j-1}}{1-z^{-1}}\Sigma(z)\frac{\mathrm d z}z,
\end{align}
where 
\[ 
\Sigma(z):=\sum_{\substack{(h,d)=1\\
p|h\Rightarrow p\in I_j}}\frac1{h^{1+s_1+s_2}}\sum_{p|m_1m_2\Rightarrow p\in I_j}\frac{\chi_d(m_1m_2)}{m_1^{s_1+1/2}m_2^{s_2+1/2}}\sum_{\substack{p|n\Rightarrow p\in I_j\\
m_1m_2|n}}\frac{z^{\Omega(n)}\lambda(n)\chi_d(n)\nu_4(n;\ell_j)}{2^{\Omega(n)}\sqrt n}c_n(m_1m_2).
\]
After changing the order of summation, the inner double sum can be written as an Euler product
\begin{align*}
&\sum_{p|n\Rightarrow p\in I_j}\frac{z^{\Omega(n)}\lambda(n)\chi_d(n)\nu_4(n)}{2^{\Omega(n)}\sqrt n}\sum_{\substack{m_1,m_2\\
m_1m_2|n}}\frac{\chi_d(m_1m_2)}{m_1^{s_1+1/2}m_2^{s_2+1/2}}c_n(m_1m_2)\\
&=\prod_{\substack{p\in I_j\\
p\nmid d}}\left(1-\frac{z}{\sqrt p}\left(\frac1{p^{s_1+1/2}}+\frac1{p^{s_2+1/2}}\right)+\frac{3z^2}{2p}\left(1+\frac1{p^{1+2s_1}}+\frac1{p^{1+2s_2}}+\frac1{p^{1+s_1+s_2}}\right)+O\left(\frac1{p^{3/2+\min\{\Re(s_1),\Re(s_2)\}}}\right)\right)
\end{align*} 
Recalling the definition of $\theta_j$, a standard application of the prime number theorem shows that (for $j\neq 0$) 
\[ 
\sum_{p\in I_j}\frac1{p^{1+s}}\ll 1
\]
uniformly in $j$ for, say, $\Re(s)\geq -(\log\log K)^2/\log K$ and $|\Im(s)|\leq e^{\sqrt{\log K}}$. From this it is easily seen that $|\Sigma(z)|\ll 1$ uniformly for $|z|\leq 2$ when $\min\{\Re(s_1),\Re(s_2)\}\geq -(\log\log K)^2/\log K$ and $\max\{|\Im(s_1)|,|\Im(s_2)|\}\leq e^{\sqrt{\log K}}$ as
\[
\sum_{\substack{(h,d)=1\\
p|h\Rightarrow p\in I_j}}\frac1{h^{1+s_1+s_2}}=\prod_{\substack{p\in I_j\\
p\nmid d}}\left(1-\frac1{p^{1+s_1+s_2}}\right)^{-1}.
 \]
Applying this bound for $\Sigma(z)$ in (\ref{Integral-over-circle}) we see that 
\[ 
\mathbb E\left(M_{2,j}(X;d)\prod_{p\in I_j}L_p\left(s_1+\frac12,s_2+\frac12;X,d\right)\right)\ll 1.
\]
Thus we have that 
\begin{align}\label{loglog-bound}
\mathbb E\left(\prod_{K^{\theta_0}<p<K^{\theta_J}}L_p\left(s_1+\frac12,s_2+\frac12;X,d\right)\prod_{j=1}^J M_{2,j}(X;d)\right)\ll (\log\log K)^{O(1)}
\end{align}
for $\min\{\Re(s_1),\Re(s_2)\}\geq -(\log\log K)^2/\log K$ and $\max\{|\Im(s_1)|,|\Im(s_2)|\}\leq e^{\sqrt{\log K}}$ by applying the above bound for each $1\leq j\leq J$ and recalling that $J\asymp\log\log\log K$. 

\subsection{Contribution of the small primes} Next we need to understand the contribution of the primes with $c\leq p\leq K^{\theta_0}$. This involves understanding the interaction between $L(s_1+1/2,s_2+1/2;X,d)$ and $M_{2,0}(X;d)$. A key point is that since $M_{2,0}(X;d)$ consists of relatively small primes we can express it in terms of an Euler product with negligible loss since $\ell_0$ is large. This allows us to simplify our later analysis by reducing the problem to understanding the contribution from each prime $p\in I_0$ individually. Let 
\begin{align}\label{modified-zero}
\widetilde M_{2,0}(X;d):=\sum_{p|n\Rightarrow p\in I_0}\frac{\lambda(n)\chi_d(n)\nu_4(n;\ell_0)}{2^{\Omega(n)}\sqrt n}\sum_{u|n}c_n(u)X(u).
\end{align}
We have the following result.
\begin{lemma}\label{small-primes}
For $\min\{\Re(s_1),\Re(s_2)\}\geq -(\log\log K)^2/\log K$ we have that 
\begin{align*}
&\mathbb E\left(M_{2,0}(X;d)\prod_{p\in I_0}L_p\left(s_1+\frac12,s_2+\frac12;X,d\right)\right)\\
&\qquad=\mathbb E\left(\widetilde M_{2,0}(X;d)\prod_{p\in I_0}L_p\left(s_1+\frac12,s_2+\frac12;X,d\right)\right)+O\left((\log K)^{-10}\right).
\end{align*}
\end{lemma}

\begin{proof}
Let us write
\[ 
R(X;d):=\widetilde M_{2,0}(X;d)-M_{2,0}(X;d)=\sum_{\substack{p|n\Rightarrow p\in I_0\\
\Omega(n)>4\ell_0}}\frac{\lambda(n)\chi_d(n)\nu_4(n;\ell_j)}{2^{\Omega(n)}\sqrt n}\sum_{u|n}c_n(u)X(u)
\]
and let us denote $L_0(s_1,s_2;X,d):=\prod_{p\in I_0}L_p(s_1,s_2;X,d)$. Then by the Cauchy--Schwarz inequality we have
\[ 
\mathbb E\left(\left|L_0\left(s_1+\frac12,s_2+\frac12;X,d\right)R(X;d)\right|\right)^2\leq\mathbb E\left(\left|L_0\left(s_1+\frac12,s_2+\frac12;X,d\right)\right|^2\right)\mathbb E\left(|R(X;d)|^2\right).
\]
Since $L_0(s_1,s_2;X,d)$ is a finite product, this function is analytic for $\min\{\Re(s_1),\Re(s_2)\}>0$. Let us first analyse $\mathbb E(|L_0(s_1+1/2,s_2+1/2;X,d)|^2)$. We simply compute
\begin{align*}
\mathbb E\left(\left|L_0\left(s_1+\frac12,s_2+\frac12;X,d\right)\right|^2\right)&=\prod_{p\in I_0}\mathbb E\left(\left|\sum_{j_1=0}^\infty\sum_{j_2=0}^\infty\frac{X(p^{j_1})X(p^{j_2})\chi_d(p^{j_1+j_2})}{p^{j_1(s_1+1/2)+j_2(s_2+1/2)}}\right|^2\right).
\end{align*}
Set 
\[ 
\rho(X,p^\alpha;s_1,s_2)):=\chi_d(p^\alpha)\sum_{j_1+j_2+j_3+j_4=\alpha}\frac{X(p^{j_1})X(p^{j_2})X(p^{j_3})X(p^{j_4})}{p^{(j_1+j_2)(s_1+1/2)+(j_3+j_4)(s_2+1/2)}}.
\]
An easy computation shows that 
\begin{align*}
\mathbb E(\rho(X,p;s_1,s_2)))=0 \qquad \text{and} \qquad \mathbb E(\rho(X,p^2;s_1,s_2))&=\chi_d(p^2)\left(\frac1{p^{1+2s_1}}+\frac1{p^{1+2s_2}}+\frac4{p^{1+s_1+s_2}}\right)\\
&\ll p^{-1-2\min\{\Re(s_1),\Re(s_2)\}}.
\end{align*}
Thus 
\[ 
\mathbb E\left(\left|L_0\left(s_1+\frac12,s_2+\frac12;X,d\right)\right|^2\right)=\prod_{p\in I_0}\left(1+O\left(\frac1{p^{1+2\min\{\Re(s_1),\Re(s_2)\}}}\right)\right).
\]
Using that $\min\{\Re(s_1),\Re(s_2)\}\geq -(\log\log K)^2/\log K$ we have $\frac1{p^{1+2\min\{\Re(s_1),\Re(s_2)\}}}\ll\frac1p$ for $p\in I_0$ and so the right-hand side of the previous display is $\ll (\log K)^{O(1)}$.

We next estimate $\mathbb E(|R(X;d)|^2)$, which equals
\begin{align*}
\sum_{\substack{p|n_1\Rightarrow p\in I_0\\
\Omega(n_1)>4\ell_0}}\sum_{\substack{p|n_2\Rightarrow p\in I_0\\
\Omega(n_2)>4\ell_0}}\frac{\lambda(n_1)\lambda(n_2)\chi_d(n_1n_2)\nu_4(n_1;\ell_0)\nu_4(n_2;\ell_0)}{2^{\Omega(n_1)+\Omega(n_2)}\sqrt{n_1n_2}}\sum_{u_1|n_1}\sum_{u_2|n_2}c_{n_1}(u_1)c_{n_2}(u_2)\mathbb E(X(u_1)X(u_2)). 
\end{align*}
We have $2^{\Omega(n_1)-4\ell_0}\geq 1$ for $\Omega(n_1)>4\ell_0$. Thus using the bound $c_n(u)\leq 2^{\Omega(n)}$ and estimating the Liouville function and $\chi_d$ trivially, we see that there exists a constant $C>0$ so that 
\begin{align}\label{error-estimate}
\mathbb E(|R(X;d)|^2)&\leq\frac1{2^{8\ell_0}}\sum_{p|r\Rightarrow p\in I_0}\sum_{p|n_1\Rightarrow p\in I_0}\sum_{p|n_2\Rightarrow p\in I_0}\frac{\lambda(rn_1)\lambda(rn_2)\chi_d(r^2n_1n_2)\nu(rn_1)\nu(rn_2)}{r\sqrt{n_1n_2}}c_{rn_1}(r)c_{rn_2}(r) \nonumber\\
&\ll\frac{(\log K)^{O(1)}}{2^{8\ell_0}}\sum_{p|r\Rightarrow p\in I_0}\frac{C^{\Omega(r)}}r \nonumber\\
&\ll\frac{(\log K)^{O(1)}}{2^{8\ell_0}},
\end{align}
where we have used the fact that $c$ is sufficiently large so that the sum in the middle converges. Combining the two estimates above and recalling that $\eta_1>0$ in the definition of $\ell_0$ is taken to be sufficiently large completes the proof.
\end{proof}

\subsection{Estimates for the large primes}
It remains to study the contribution of the primes $p>K^{\theta_J}$. These do not interact with our mollifier and consequently their contribution is easy to understand. Estimating these terms precisely allows us to meromorphically continue $\mathbb E(L(s_1+1/2,s_2+1/2,X;d)M_2(X;d))$ as $M_2(X;d)$ is a Dirichlet polynomial with coefficients supported on integers with prime factors $\leq K^{\theta_J}$. 

\begin{lemma}
For $\min\{\Re(s_1),\Re(s_2)\}>1/2$ we have that 
\[ 
\mathbb E\left(L_p\left(s_1+\frac12, s_2+\frac12;X,d\right)\right)=\begin{cases}
\zeta_p(1+s_1+s_2) &\text{if }p\nmid d \\
1 & \text{if }p|d
\end{cases} 
\]
\end{lemma}

\begin{proof}
We simply compute
\begin{align*}
\mathbb E\left(L_p\left(s_1+\frac12,s_2+\frac12;X,d\right)\right)&=\sum_{k_1=0}^\infty\sum_{k_2=0}^\infty\frac{\chi_d(p^{k_1+k_2})}{p^{k_1(s_1+1/2)+k_2(s_2+1/2)}}\mathbb E\left(X(p^{k_1})X(p^{k_2})\right)\\
&=\sum_{j=0}^\infty\frac{\chi_d(p^{2j})}{p^{(1+s_1+s_2)j}}\\
&=1+1_{p\nmid d}\cdot\sum_{j=1}^\infty\frac1{p^{(1+s_1+s_2)j}}\\
&=1+1_{p\nmid d}\cdot\left(\frac1{1-p^{-(1+s_1+s_2)}}-1\right)\\
&=1+1_{p\nmid d}\cdot\left(\zeta_p(1+s_1+s_2)-1\right),
\end{align*}
which proves the claim.
\end{proof}
\noindent Thus using the dominated convergence theorem we have for $\min\{\Re(s_1),\Re(s_2)\}>1/2$ and any $z\geq 2$ that
\begin{align*}
\mathbb E\left(\prod_{p>z}L_p\left(s_1+\frac12,s_2+\frac12;X,d\right)\right)&=\prod_{p>z}\mathbb E\left(L_p\left(s_1+\frac12,s_2+\frac12;X,d\right)\right)\\
&=\prod_{\substack{p>z\\
p\nmid d}}\zeta_p(1+s_1+s_2)\\
&=\zeta(1+s_1+s_2)\prod_{p\leq z}\zeta_p(1+s_1+s_2)^{-1}\prod_{\substack{p>z\\
p|d}}\zeta_p(1+s_1+s_2)^{-1}. 
\end{align*}
If $u$ is a natural number with all prime factors $\leq K^{\theta_J}$, this provides a meromorphic continuation for $\mathbb E(X(u)L(s_1+1/2,s_2+1/2;X,d))$ to $\Re(s_1+s_2)>-1/2$ with simple poles along the hyperplane $s_1+s_2=0$. Indeed, by choosing $z=K^{\theta_J}$ above, we have
\begin{align}\label{twisted-exp}
&\mathbb E\left(X(u)L(s_1+1/2,s_2+1/2;X,d)\right) \nonumber\\
&=\mathbb E\left(X(u)\prod_{p\leq K^{\theta_J}}L_p\left(s_1+\frac12,s_2+\frac12;X,d\right)\right)\mathbb E\left(\prod_{p>K^{\theta_J}}L_p\left(s_1+\frac12,s_2+\frac12;X,d\right)\right)\nonumber\\
&=\zeta(1+s_1+s_2)\prod_{p\leq K^{\theta_J}}\zeta_p(1+s_1+s_2)^{-1}\prod_{\substack{p>K^{\theta_J}\\
p|d}}\zeta_p(1+s_1+s_2)^{-1}\mathbb E\left(X(u)\prod_{p\leq K^{\theta_J}}L_p\left(s_1+\frac12,s_2+\frac12;X,d\right)\right).
\end{align} 
We further compute using the Hecke relations for $X(n)$ and (\ref{correlation-expectation}) that 
\begin{align}\label{twisted-exp-small}
\mathbb E\left(X(u)\prod_{p\leq K^{\theta_J}}L_p\left(s_1+\frac12,s_2+\frac12;X,d\right)\right)&=\mathbb E\left(X(u)\sum_{\substack{m_1,m_2\\
p|m_1m_2\Rightarrow p\leq K^{\theta_J}}}\frac{X(m_1)X(m_2)\chi_d(m_1m_2)}{m_1^{s_1+1/2}m_2^{s_2+1/2}}\right) \nonumber\\
&=\sum_{\substack{m_1,m_2\\
p|m_1m_2\Rightarrow p\leq K^{\theta_J}}}\frac{\chi_d(m_1m_2)}{m_1^{s_1+1/2}m_2^{s_2+1/2}}\sum_{h|(m_1,m_2)}1_{u=m_1m_2/h^2} \nonumber\\
&=\sum_{\substack{(h,d)=1\\
p|h\Rightarrow p\leq K^{\theta_J}}}\frac1{h^{1+s_1+s_2}}\sum_{\substack{m_1,m_2\\
p|m_1m_2\Rightarrow p\leq K^{\theta_J}}}\frac{\chi_d(m_1m_2)}{m_1^{s_1+1/2}m_2^{s_2+1/2}}\cdot 1_{u=m_1m_2} \nonumber\\
&=\frac{\chi_d(u)}{u^{1/2+s_2}}\prod_{\substack{p\leq K^{\theta_J}\\
p\nmid d}}\left(1-\frac1{p^{1+s_1+s_2}}\right)^{-1}\sum_{v|u}\frac1{v^{s_1-s_2}}.
\end{align}
Combining (\ref{twisted-exp}) and (\ref{twisted-exp-small}) we have 
\begin{align*}
&\mathbb E\left(X(u)\prod_{p\leq K^{\theta_J}}L_p\left(s_1+\frac12,s_2+\frac12;X,d\right)\right)\\
&=\zeta(1+s_1+s_2)\frac{\chi_d(u)}{u^{1/2+s_2}}\prod_{\substack{p\leq K^{\theta_J}\\
p\nmid d}}\left(1-\frac1{p^{1+s_1+s_2}}\right)^{-1}\prod_{p\leq K^{\theta_J}}\left(1-\frac1{p^{1+s_1+s_2}}\right)\prod_{\substack{p>K^{\theta_J}\\
p|d}}\left(1-\frac1{p^{1+s_1+s_2}}\right)\sum_{v|u}\frac1{v^{s_1-s_2}}\\
&=\zeta(1+s_1+s_2)\frac{\chi_d(u)}{u^{1/2+s_2}}\prod_{p|d}\left(1-\frac1{p^{1+s_1+s_2}}\right)\sum_{v|u}\frac1{v^{s_1-s_2}},
\end{align*}
which provides the promised meromorphic continuation to $\Re(s_1+s_2)>-1/2$.

Let us define 
\[ 
\breve h_2(n):=\sum_{\substack{n_0\cdots n_J=n \\
p|n_j\Rightarrow p\in I_j\,\forall\, 0\leq j\leq J\\
\Omega(n_j)\leq 4\ell_j\,\forall 1\leq j\leq J}}\nu_{4}(n_0;\ell_0)\cdots\nu_{4}(n_J;\ell_J).
\]
Our next goal is to prove the following result. 

\begin{lemma}\label{random-averaged}
Suppose that $X\ll\sqrt K$. Then we have
\begin{align*}
&\sumflat_{d}\,\phi\left(\frac{|d|}X\right)\sum_{k\in\Z}h\left(\frac{2k-1}K\right)\mathbb E\left(L(X;d,k)M_2(X;d)\right)\\
&=K(\log K)\widehat h(0)\sum_{n\leq K^{4\delta_0}}\frac{\breve h_2(n)\lambda(n)}{2^{\Omega(n)}\sqrt n}\sumflat_{d}\,\phi\left(\frac{|d|}X\right)\frac{\varphi(|d|)}{|d|}\sum_{u|n}\frac{c_n(u)d(u)}{\sqrt u}\chi_d(nu)\left(\log\left(\frac{K|d|}{2\pi\sqrt u}\right)+\sum_{p|d}\frac{\log p}{p-1}+\gamma\right)\\
&-K(\log K)\widehat h(0)\sum_{n\leq K^{4\delta_0}}\frac{h_2(n)\lambda(n)}{2^{\Omega(n)}\sqrt n}\sumflat_{d}\,\phi\left(\frac{|d|}X\right)\frac{\varphi(|d|)}{|d|}\sum_{u|n}\frac{c_n(u)}{\sqrt u}\chi_d(nu)\frac1{2\pi i}\int\limits_{(1+\eps)}e^{2s_1^2}u^{-s_1}\sum_{v|u}v^{2s_1}\,\frac{\mathrm d s_1}{s_1^2}\\&\qquad\qquad\qquad\qquad\qquad\qquad+O\left(KX(\log K)^{-7}\right).
\end{align*}  
\end{lemma}

\begin{proof} By Mellin inversion we have
\begin{align}\label{before-shift}
&\sumflat_{d}\phi\left(\frac{|d|}X\right)\sum_{k\in\Z}h\left(\frac{2k-1}K\right)\mathbb E\left(L(X;d,k)M_2(X;d)\right) \nonumber\\
&=2\sumflat_{d}\phi\left(\frac{|d|}X\right)\sum_{k\in\Z}h\left(\frac{2k-1}K\right)(1+\varepsilon_{k,d}) \nonumber\\
&\qquad\qquad\times\mathbb E\left(\frac1{(2\pi i)^2}\int\limits_{(\frac12+\eps)}\int\limits_{(\frac12+\eps)}\left(\frac{2\pi}{|d|}\right)^{-s_1-s_2}\frac{\Gamma(s_1+k)}{\Gamma(k)}\cdot\frac{\Gamma (s_2+k)}{\Gamma(k)}e^{s_1^2+s_2^2}F(s_1,s_2;X,d)\frac{\mathrm d s_1\,\mathrm d s_2}{s_1s_2}\right), 
\end{align}
where 
\[ 
F(s_1,s_2;X,d):=M_2(X;d)\sum_{m_1=1}^\infty\sum_{m_2=1}^\infty\frac{X(m_1)X(m_2)\chi_d(m_1m_2)}{m_1^{s_1+1/2}m_2^{s_2+1/2}}=M_2(X;d)L\left(s_1+\frac12,s_2+\frac12;X,d\right).
\]
Using the definition of $M_2(X;d)$ it follows from the discussion above that $\mathbb E(F(s_1,s_2;X,d))$ has simple poles along the hyperplane $s_1+s_2=0$. 

We move the line of integration from $\Re(s_2)=1/2+\eps$ to $\Re(s_2)=-1$, crossing simple poles at $s_2=0$ and $s_2=-s_1$. Clearly the double line integral is $O_\eps(X^{1/2}K^{1/2+\varepsilon})$ after shifting the contour. Next we shall study the contribution of the residues at these poles to (\ref{before-shift}) separately. \\

\noindent $\bullet$ \underline{Residue at $s_2=0$:} Using (\ref{twisted-exp}) this contribution is given by
\begin{align}\label{s1-int}
&2\sumflat_{d}\phi\left(\frac{|d|}X\right)\sum_{k\in\Z}h\left(\frac{2k-1}K\right)(1+\eps_{k,d})\frac1{2\pi i}\int\limits_{(\frac12+\eps)}\left(\frac{2\pi}{|d|}\right)^{-s_1}\frac{\Gamma(s_1+k)}{\Gamma(k)}e^{s_1^2}\zeta(1+s_1)\prod_{p\leq K^{\theta_J}}\zeta_p(1+s_1)^{-1}\\
&\times\prod_{\substack{p>K^{\theta_J}\\
p|d}}\zeta_p(1+s_1)^{-1}\cdot\mathbb E\left(M_{2,0}(X;d)\prod_{p\in I_0}L_p\left(s_1+\frac12,\frac12;X,d\right)\right)\mathbb E\left(\prod_{j=1}^J M_{2,j}(X;d)\prod_{p\in I_j}L_p\left(s_1+\frac12,\frac12;X,d\right)\right)\,\frac{\mathrm d s_1}{s_1}. \nonumber
\end{align}
Recalling (\ref{loglog-bound}) and repeating the arguments leading to it we see that 
\begin{align}\label{estimate-s2-zero}
\mathbb E\left(M_2(X;d)L\left(s_1+\frac12,\frac12;X,d\right)\right)\ll (\log K)^{O(1)}
\end{align}
in the region $\Re(s_1)\geq\frac12+\eps$. Since for fixed $\sigma>0$ Stirling's formula gives $|\Gamma(\sigma+it)|\ll (|t|+1)^{\sigma-\frac12}e^{-\pi|t|/2}$, by (\ref{estimate-s2-zero}) we may truncate the $s_1$-integral (\ref{s1-int}) to $|\Im(s_1)|\leq B\log K$ at the cost of an error term of size $O(X)$, where $B$ is a sufficiently large absolute constant. 
 
We then shift the line of integration in (\ref{s1-int}) to $\Re(s_1)=-(\log\log K)^2/\log K$, crossing a double pole at $s_1=0$. The integral on the new line is, say, $O(KX(\log K)^{-10})$ and the horizontal lines give a much smaller contribution when $B$ is chosen to be sufficiently large. Using Lemma \ref{small-primes} and (\ref{loglog-bound}) we may replace the factor 
\[ 
\mathbb E\left(M_{2,0}(X;d)\prod_{p\in I_0}L_p\left(\frac12,\frac12;X,d\right)\right) \quad\text{by}\quad \mathbb E\left(\widetilde M_{2,0}(X;d)\prod_{p\in I_0}L_p\left(\frac12,\frac12;X,d\right)\right)
\]
with an error, say, $O(KX(\log K)^{-7})$, where we recall the definition of $\widetilde M_{2,0}(X;d)$ from (\ref{modified-zero}).

A straightforward computation, using the residue theorem and (\ref{twisted-exp-small}), shows that the double pole contributes the amount
\begin{align*}
&2(\log K)\sum_{n\leq K^{4\delta_0}}\frac{\breve h_2(n)\lambda(n)}{2^{\Omega(n)}\sqrt n}\sumflat_{d}\,\phi\left(\frac{|d|}X\right)\frac{\varphi(|d|)}{|d|}\sum_{u|n}\frac{c_n(u)\chi_d(nu)}{\sqrt u}\\
&\qquad\qquad\times\sum_{k\in\Z}h\left(\frac{2k-1}K\right)(1+\eps_{k,d})\left(-\log\left(\frac{2\pi}{|d|}\right)\cdot d(u)+\frac{\Gamma'(k)}{\Gamma(k)}\cdot d(u)+\sum_{p|d}\frac{\log p}{p-1}\cdot d(u)-\sum_{v|u}\log v+\gamma\cdot d(u)\right)
\end{align*} 
to (\ref{s1-int}) as
\begin{align*}
&\mathbb E\left(\widetilde M_{2,0}(X;d)\prod_{p\in I_0}L\left(\frac12,\frac12;X,d\right)\prod_{j=1}^J M_{2,j}(X;d)\prod_{p\in I_j}L\left(\frac12,\frac12;X,d\right)\right)\\
&=\sum_{n\leq K^{4\delta_0}}\frac{\widetilde h_2(n)\lambda(n)\chi_d(n)}{2^{\Omega(n)}\sqrt n}\sum_{u|n}\frac{c_n(u)\chi_d(u)d(u)}{\sqrt u}\prod_{\substack{p\leq K^{\theta_J}\\
p\nmid d}}\left(1-\frac1p\right)^{-1}.
\end{align*}
Here $\gamma$ is the Euler--Mascheroni constant. 

The expression above can be simplified by using the easy observation that 
\[\frac{\Gamma'(k)}{\Gamma(k)}=\log(k-1)+O\left(\frac 1k\right)=\log\left(K\cdot\frac{k-1}K\right)+O\left(\frac 1k\right),
\]
and noting that 
\[ 
\sum_{v|u}\log v=\frac12 d(u)\log u\]
to be
\begin{align*}
&2(\log K)\sum_{n\leq K^{4\delta_0}}\frac{\breve h_2(n)\lambda(n)}{2^{\Omega(n)}\sqrt n}\sumflat_{d}\,\phi\left(\frac{|d|}X\right)\frac{\varphi(|d|)}{|d|}\sum_{u|n}\frac{c_n(u)d(u)}{\sqrt u}\chi_d(nu)\\
&\qquad\qquad\qquad\times\sum_{k\in\Z}h\left(\frac{2k-1}K\right)(1+\varepsilon_{k,d})\left(\log\left(\frac{k|d|}{2\pi\sqrt u}\right)+\sum_{p|d}\frac{\log p}{p-1}+\gamma\right)+O_\eps\left(XK^{2\delta_0+\eps} \right).
\end{align*}
Then evaluating the $k$-sum by Poisson summation, the above main contribution may be written, up to an error $O(XK^{1/2})$, as
\begin{align*}
K(\log K)\widehat h(0)\sum_{n\leq K^{4\delta_0}}\frac{\breve h_2(n)\lambda(n)}{2^{\Omega(n)}\sqrt n}\sumflat_{d}\,\phi\left(\frac{|d|}X\right)\frac{\varphi(|d|)}{|d|}\sum_{u|n}\frac{c_n(u)d(u)}{\sqrt u}\chi_d(nu)\left(\log\left(\frac{K|d|}{2\pi\sqrt u}\right)+\sum_{p|d}\frac{\log p}{p-1}+\gamma\right).
\end{align*}

\noindent $\bullet$ \underline{Residue at $s_2=-s_1$:} Likewise, using the series representation 
\[\zeta(1+x)=\frac1{x}+\gamma+\cdots,\]
the term corresponding to this residue equals 
\begin{align*}
&-2(\log K)\sum_{n\leq K^{4\delta_0}}\frac{h_2(n)\lambda(n)}{2^{\Omega(n)}\sqrt n}\sumflat_{d}\,\phi\left(\frac{|d|}X\right)\frac{\varphi(|d|)}{|d|}\sum_{u|n}\frac{c_n(u)}{\sqrt u}\chi_d(nu)\\
&\qquad\qquad\times\sum_{k\in\Z}h\left(\frac{2k-1}K\right)(1+\varepsilon_{k,d})\frac1{2\pi i}\int\limits_{(1+\eps)}\frac{\Gamma(s_1+k)}{\Gamma(k)}\cdot\frac{\Gamma(-s_1+k)}{\Gamma(k)}e^{2s_1^2}u^{-s_1}\sum_{v|u}v^{2s_1}\,\frac{\mathrm d s_1}{s_1^2}. 
\end{align*}
Executing the $k$-sum by Poisson summation, it follows that this contribution is given by  
\begin{align}\label{diagonal-second-part}
&-K(\log K)\widehat h(0)\sum_{n\leq K^{4\delta_0}}\frac{h_2(n)\lambda(n)}{2^{\Omega(n)}\sqrt n}\sumflat_{d}\,\phi\left(\frac{|d|}X\right)\frac{\varphi(|d|)}{|d|}\sum_{u|n}\frac{c_n(u)}{\sqrt u}\chi_d(nu)\frac1{2\pi i}\int\limits_{(1+\eps)}e^{2s_1^2}u^{-s_1}\sum_{v|u}v^{2s_1}\,\frac{\mathrm d s_1}{s_1^2}
\end{align}
up to a negligible error. This completes the proof.
\end{proof}

\section{Proof of Proposition \ref{fourth-moment-with-average}}

\noindent By combining Corollary \ref{mollified-moment-connection} and Lemma \ref{random-averaged} we immediately get the following result. Here we note that (\ref{diagonal-second-part}) precisely cancels (\ref{off-diag-main}).

\begin{lemma}
Suppose that $X\ll\sqrt K$. Then we have
\begin{align*}
&\sumflat_{d}\phi\left(\frac{|d|}X\right)\sum_{k\in\Z}h\left(\frac{2k-1}K\right)\sum_{f\in B_k}\omega_f L\left(\frac12,f\otimes\chi_d\right)^2M_f(d)^4\\
&=K(\log K)\widehat h(0)\sum_{n\leq K^{4\delta_0}}\frac{\breve h_2(n)\lambda(n)}{2^{\Omega(n)}\sqrt n}\sumflat_{d}\,\phi\left(\frac{|d|}X\right)\frac{\varphi(|d|)}{|d|}\sum_{u|n}\frac{c_n(u)d(u)}{\sqrt u}\chi_d(nu)\left(\log\left(\frac{K|d|}{2\pi\sqrt u}\right)+\sum_{p|d}\frac{\log p}{p-1}+\gamma\right)\\
&\qquad\qquad\qquad\qquad\qquad\qquad+O\left(KX(\log K)^{-7}\right).
\end{align*}
\end{lemma}
\noindent At this point we need to evaluate the sum
\begin{align}\label{char-sum-w-phi}
\sumflat_{d}\chi_d(nu)\frac{\varphi(|d|)}{|d|}\log\left(\frac{K|d|}{2\pi\sqrt u}\right)\phi\left(\frac{|d|}X\right).
\end{align}
Writing 
\[ 
\frac{\varphi(|d|)}{|d|}=\sum_{j|d}\frac{\mu(j)}j
\]
and using (\ref{sqf-detection}) together with Lemma \ref{Poisson-w-char} we see that (\ref{char-sum-w-phi}) equals
\begin{align*}
\frac X4\sum_{\substack{\alpha=1\\
(\alpha,2nu)=1}}^\infty\frac{\mu(\alpha)}{\alpha^2}\sum_{(j,2)=1}\frac{\mu(j)}{j^2}\cdot\frac{\varphi(nu)}{nu}\cdot1_{nu=\square}\int\limits_0^\infty\phi(y)\log\left(\frac{XKy}{2\pi\sqrt u}\right)\,\mathrm d y+O_\eps\left(X^{1/2+\eps}\sqrt{nu}\right).
\end{align*}
Using the previous estimate to bound the $d$-sum and recalling (\ref{simplification}) we have that, when $\delta_0$ is sufficiently small, 
\begin{align}\label{mollified-bound}
&\sumflat_{d}\,\phi\left(\frac{|d|}X\right)\sum_{k\in\Z}h\left(\frac{2k-1}K\right)\sum_{f\in B_k}\omega_f  L\left(\frac12,f\otimes\chi_d\right)^2 M_f(d)^4 \nonumber \\
&\ll_\eps XK(\log K)^2\sum_{n\leq K^{4\delta_0}}\frac{\breve h_2(n)\lambda(n)}{2^{\Omega(n)}\sqrt n}\sum_{\substack{u|n\\
nu=\square}}\frac{c_n(u)d(u)}{\sqrt u}\prod_{p|n}\left(1+\frac1p\right)^{-1}+KX(\log K)^{-7}+X^{1/2+\eps}K^{1+4\delta_0} \nonumber\\
&= XK(\log K)^2\left(\sum_{p|n\Rightarrow p\in I_0}\frac{\lambda(n)\nu_4(n;\ell_j)}{2^{\Omega(n)}\sqrt n}\sum_{\substack{u|n\\
nu=\square}}\frac{c_n(u)d(u)}{\sqrt u}\prod_{p|n}\left(1+\frac1p\right)^{-1}\right) \nonumber\\
&\qquad\qquad\times\prod_{j=1}^J\sum_{\substack{p|n\Rightarrow p\in I_j\\
\Omega(n)\leq 4\ell_j}}\frac{\lambda(n)\nu_4(n;\ell_j)}{2^{\Omega(n)}\sqrt n}\sum_{\substack{u|n\\
nu=\square}}\frac{c_n(u)d(u)}{\sqrt u}\prod_{p|n}\left(1+\frac1p\right)^{-1}+O\left(KX(\log K)^{-7}\right).
\end{align}
Let us now estimate the first term on the right-hand side of (\ref{mollified-bound}). We remove the condition $\Omega(n)\leq 4\ell_j$ by arguing as before: using Rankin's trick and the estimate $c_n(u)\leq 2^{\Omega(n)}$ it is easy to see (as in the derivation of (\ref{error-estimate})) that there exists $C>0$ such that 
\begin{align*}
&\sum_{\substack{p|n\Rightarrow p\in I_j\\
\Omega(n)\leq 4\ell_j}}\frac{\lambda(n)\nu_4(n;\ell_j)}{2^{\Omega(n)}\sqrt n}\sum_{\substack{u|n\\
nu=\square}}\frac{c_n(u)d(u)}{\sqrt u}\prod_{p|n}\left(1+\frac1p\right)^{-1}\\
&=\sum_{p|n\Rightarrow p\in I_j}\frac{\lambda(n)\nu_4(n)}{2^{\Omega(n)}\sqrt n}\sum_{\substack{u|n\\
nu=\square}}\frac{c_n(u)d(u)}{\sqrt u}\prod_{p|n}\left(1+\frac1p\right)^{-1}+O\left(\frac1{2^{\ell_j}}\sum_{p|m\Rightarrow p\in I_j}\frac{C^{\Omega(m)}}m\right)
\end{align*}
Here we have also again used the properties of $\nu_r(n;\ell)$ listed after (\ref{Mollifier-power}). The error term is clearly $\ll 2^{-\ell_j}$. For the main term we use the facts $\nu_4(p)=4$, $\nu_4(p^2)=8$, $c_p(p)=1$ and $c_{p^2}(p)=0$ to see that it equals
\[ 
\prod_{p\in I_j}\left(1-\frac2p+O\left(\frac1{p^{3/2}}\right)\right).\]
It follows that the product over $1\leq j\leq J$ is
\[ 
\left(1+O\left(K^{-\theta_0}\right)+O\left(\sum_{j=1}^J\frac1{2^{\ell_j}}\right)\right)\prod_{j=1}^J\prod_{p\in I_j}\left(1-\frac2p+O\left(\frac1{p^{3/2}}\right)\right).
\]
Using that $2^{-\ell_j}\ll 1/\ell_j\asymp\theta_j^{3/4}$, and summing the geometric series the error term is $\ll \theta_J^{3/4}\ll 1$. Hence, it follows that the right-hand side of (\ref{mollified-bound}) is
\[ 
\ll XK(\log K)^2\prod_{c<p\leq K^{\theta_J}}\left(1-\frac2p+O\left(\frac1{p^{3/2}}\right)\right).
\]
\noindent But this is $\ll XK(\log K)^2(\log K)^{-2}\ll XK$ by Mertens' theorem. This completes the proof of Proposition \ref{fourth-moment-with-average}. \qed

\section{Proof of Theorem \ref{Main-theorem}}

\noindent In this section we complete the proof of Theorem \ref{Main-theorem}. Towards this the main observation is the following result, which connects finding real zeroes to (essentially) detecting sign changes among the Fourier coefficients. 

\begin{prop}(\cite[Proposition 6.1.]{Jaasaari2026})\label{Fourier-coeff}
Let $\alpha\in\{-\frac12,0\}$. Then there are positive constants $c_1,c_2$ and $\eta$ such that, for all integers $\ell\in]c_1,c_2\sqrt{k/\log k}[$ and all Hecke eigenforms $g\in S^+_{k+\frac12}(4)$, we have
\begin{align*}
\sqrt{\alpha_g}\left(\frac e{\ell}\right)^{\frac{k}2-\frac14}g(\alpha+iy_\ell)=\sqrt{\alpha_g}c_g(\ell)e(\alpha\ell)+O(k^{-1/2-\eta}),
\end{align*}
where $y_\ell:=(k-1/2)/4\pi\ell$ and the implicit constant in the error term is absolute. 
\end{prop}
 Let $\alpha\in\{-\frac12,0\}$, $\varepsilon>0$ be any fixed small constant, and let $K$ be a large positive parameter. Recall that as $g(\alpha+iy)$ is real-valued for these values of $\alpha$, Proposition \ref{Fourier-coeff} yields information on the zeroes inside the Siegel sets
\[\mathcal F_Y=\{z\in\mathcal F:\, \Im(z)\geq Y\}\]
with $c_1'\sqrt{k\log k}\leq Y\leq c_2'k$ for some positive constants $c_1'$ and $c_2'$. Let $\eta$ be as in Proposition \ref{Fourier-coeff}. It follows immediately from this result that if we can find numbers $\ell_1,\ell_2\in]c_1,c_2k/Y[$ so that 
\begin{align}\label{real-zero}
\sqrt{\alpha_g}c_g(\ell_1)e(\alpha\ell_1)<-k^{-\delta}<k^{-\delta}<\sqrt{\alpha_g}c_g(\ell_2)e(\alpha\ell_2)
\end{align}
for some $1/2<\delta<1/2+\eta$, then $g(z)$ has a zero $\alpha+iy$ with $y$ between $y_{\ell_1}$ and $y_{\ell_2}$. Observe that 
\begin{align*}
e(\alpha\ell)=\begin{cases}
1 & \text{if }\alpha=0\\
(-1)^\ell & \text{if }\alpha=-1/2
\end{cases}
\end{align*}
Hence, in order to find real zeroes on the line $\Re(s)=0$ it suffices (essentially) to detect sign changes among the Fourier coefficients whereas on the line $\Re(s)=-1/2$ one needs to find pairs $(\ell_1,\ell_2)$ with $\ell_i$ odd for which (\ref{real-zero}) holds. As we restrict to odd fundamental discriminants $d$ for which $(-1)^k d>0$, we automatically obtain real zeroes on both of the individual geodesic segments $\Re(s)=-1/2$ and $\Re(s)=0$. 

Several positive constants appear throughout the proof and we start by summarising their roles. The order in which we fix the constants matters as the constants fixed later may depend on those chosen earlier. The constants $C_6$, $C_9$, $C_8$, $C_{14 }$, $C_{2}$, and $C_1$ are fixed in this order. The constants are enumerated in the order they appear in the argument, but are not fixed in the same order.  
\begin{enumerate}
\item $C_3$ is an absolute constant\footnote{Constants $C_3$, $C_4$, and $C_5$ are allowed to depend on the weight functions $\phi$, $h_1$, and $h_2$ specified later.} so that 
\[ 
\sum_{k\in\Z}h\left(\frac{2k-1}K\right)\sum_{g\in B_{k+1/2}^+}\alpha_g^2\omega_g^{-1}\sumflat_d\,|c_g(|d|)|^4M_g(d)^4\phi\left(\frac{|d|}X\right)\leq C_3 XK. 
\]
\item $C_4$ and $C_5$ are absolute constants so that
\[ 
C_4XK\leq \sum_{k\in\Z}h\left(\frac{2k-1}K\right)\sum_{g\in B_{k+1/2}^+}\alpha_g\sumflat_d\,|c_g(|d|)|^2M_g(d)^2\phi\left(\frac{|d|}X\right)\leq C_5XK.\]
\item $C_{10}$ is an absolute constant so that
\[
\sum_{g\in B_{k+\frac12}^+}\omega_g^{1/2}\leq C_{10}\sqrt k. 
\]
\item $C_{12}$ is an absolute constant so that
\[
\sum_{g\in B_{k+\frac12}^+}\omega_g^{2}\leq\frac{C_{12}}k. 
\]
\item $C_6>0$ is chosen so that $C_4^{3/2}/\sqrt{C_3}>C_5/\sqrt{C_6}$.
\item $C_7:=C_4^{3/2}/\sqrt{C_3}-C_5/\sqrt{C_6}$. 
\item $C_9$ we are able to choose freely to be sufficiently large.
\item $C_8$ is chosen so that $C_7>2C_8C_9C_{10}$. 
\item $C_{11}:=(C_7-2C_8C_9C_{10})/C_9\sqrt{C_6}$. 
\item $C_{13}:=C_{11}^2/4C_{12}$.
\item $C_{14}>0$ is chosen so that $C_{14}<4C_{13}$. 
\item $C_{15}:=C_{14}/4$. 
\item $C_2>0$ is chosen so that $C_2<C_{14}$. 
\item $C_1$ is chosen so that $C_1>C_8$ and $C_{13}-C_{15}>8\sqrt{C_3}/C_1^2C_2$.  
\end{enumerate}
\noindent Observe that the existence of the constants $C_3,C_4$, and $C_5$ follows from Propositions \ref{second-moment-with-average} and \ref{fourth-moment-with-average}, respectively. Recall that these constants and $C_{10}, C_{12}$ are absolute. 

For the rest of the paper, set $X=K/Y$. Then $K^\vartheta\ll X\ll\sqrt{K/\log K}$ by the assumptions on $Y$. Remember that in order to detect sign changes along the sequence $d\equiv 1\,(\text{mod }4)$ with $d$ squarefree and $(-1)^kd>0$ in the short interval $[x,x+H]$, $x\sim X$, it suffices to have
\[
\left|\sumflat_{x\leq (-1)^kd\leq x+H}\sqrt{\alpha_g}c_g(|d|)M_g(d)\right|<\sumflat_{x\leq (-1)^kd\leq x+H}\sqrt{\alpha_g}|c_g(|d|)|M_g(d).
\]

Now the proof can be completed following the arguments of \cite{Jaasaari2026}. Denote
\[S_{1,g}(x;H):=\left|\sumflat_{x\leq (-1)^kd\leq x+H}\sqrt{\alpha_g}c_g(|d|)M_g(d)\right|\]
and 
\[\mathcal T_{1,g}(X;H):=\left\{x\sim X:\,S_{1,g}(x;H)\geq C_1\sqrt H k^{-1/2}\right\}.\]
Choose $h_1$ to be a non-negative smooth function that is supported in the interval $[3/2,9/2]$ and is identically one in $[2-\eps,4]$ for some small fixed $\eps>0$. With the above notation we have by Markov's inequality that
\begin{align}\label{application-of-Chebyshev}
&\sum_{k\sim K}\sum_{\substack{g\in B_{k+\frac12}^+\\|\mathcal T_{1,g}(X;H)|\geq C_2 X}} 1 \nonumber\\
&\leq\frac{1}{C_2X}\sum_{k\in\Z}h_1\left(\frac{2k-1}K\right)\sum_{g\in B_{k+\frac12}^+}|\mathcal T_{1,g}(X;H)| \nonumber\\
&=\frac{1}{C_2X}\sum_{k\in\Z}h_1\left(\frac{2k-1}K\right)\sum_{g\in B_{k+\frac12}^+}\sum_{\substack{x\sim X\\
S_{1,g}(x;H)\geq C_1 \sqrt H k^{-1/2}}}1 \nonumber \\
&\leq \frac {2K}{C_2C_1^2HX}\sum_{k\in\Z}h_1\left(\frac{2k-1}K\right)\sum_{g\in B_{k+\frac12}^+}\alpha_g\sum_{x\sim X}\left|\sumflat_{x\leq (-1)^kd\leq x+H}c_g(|d|)M_g(d)\right|^2. 
\end{align}
By opening the absolute square the inner sum over $x\sim X$ can be rearranged into
\[ 
\sum_{x\sim X}\sumflat_{x\leq(-1)^kd_1\leq x+H}\sumflat_{x\leq(-1)^kd_2\leq x+H}c_g(|d_1|)c_g(|d_2|)M_g(d_1)M_g(d_2).\]
Let us first focus on the diagonal terms with $d_1=d_2$. In this case the total contribution to (\ref{application-of-Chebyshev}) is given by

\[\leq \frac {4K}{C_2C_1^2X}\sum_{k\in\Z}h_1\left(\frac{2k-1}K\right)\sum_{g\in B_{k+\frac12}^+}\alpha_g\sumflat_{(-1)^kd\sim X}|c_g(|d|)|^2M_g(d)^2.\]
Adding a smooth weight function $\phi$ that localises\footnote{Note that for $d\equiv 1\,(\text{mod }4)$, $c_g(|d|)=0$ unless $(-1)^kd=|d|$.} $(-1)^kd\sim X$ and applying the Cauchy--Schwarz inequality along with Proposition \ref{fourth-moment-with-average} this is bounded by
\begin{align*}
&\leq\frac{4K}{C_2C_1^2X}\left(\sum_{k\in\Z}h\left(\frac{2k-1}K\right)\sum_{g\in B_{k+1/2}^+}\alpha_g^2\omega_g^{-1}\sumflat_d\,|c_g(|d|)|^4M_g(d)^4\phi\left(\frac{|d|}X\right)\right)^{1/2}\left(\sum_{k\sim K}\sum_{g\in B_{k+\frac12}^+}\omega_g\sumflat_{d\sim X}1\right)^{1/2}\\
&\leq\frac{8\sqrt{C_3}K^2}{C_2C_1^2}. 
\end{align*}
The off-diagonal contribution is negligible identically as in the proof of \cite[Proposition 2.5.]{Jaasaari2026} using the trivial bound $M_g(d)\ll K^{\delta_0}(\log K)^{1/4}$. 

Next we derive a lower bound for the weighted sum of the terms $|c_g(|d|)|$ on average over the forms $g\in\mathcal S_K$. Choose a non-negative compactly supported weight function $h_2$ so that $h_2((2k-1)/K)\leq 1_{[K,2K]}(k)$. Applying H\"older's inequality as in the proof of \cite[Proposition 2.6.]{Jaasaari2026} we have 
\begin{align*}
\sum_{k\in\Z}h_2\left(\frac{2k-1}K\right)\sum_{g\in B_{k+\frac12}^+}\sumflat_{(-1)^kd\sim X}\alpha_g^{1/2}\omega_g^{1/2}|c_g(|d|)|M_g(d)\geq \frac{C_4^{3/2}}{\sqrt{C_3}}XK.
\end{align*}
On the other hand, by (\ref{W-rel}), Proposition \ref{second-moment-with-average}, and Waldspurger's formula we also have 
\begin{align*}
&\sum_{k\in\Z}h_2\left(\frac{2k-1}K\right)\sum_{g\in B_{k+\frac12}^+}\sumflat_{\substack{(-1)^kd\sim X\\
M_g(d)^2L(1/2,f\otimes \chi_{d})>C_6}}\alpha_g^{1/2}\omega_g^{1/2}|c_g(|d|)|M_g(d)\\
&\leq\frac1{\sqrt{C_6}}\sum_{k\in\Z}h_2\left(\frac{2k-1}K\right)\sum_{g\in B_{k+\frac12}^+}\sumflat_{(-1)^k d\sim X}\alpha_g|c_g(|d|)|^2M_g(d)^2\\
&\leq\frac{C_5}{\sqrt{C_6}}XK.
\end{align*}

From these we infer the lower bound
\[\sum_{k\in\Z}h_2\left(\frac{2k-1}K\right)\sum_{g\in B_{k+\frac12}^+}\sumflat_{\substack{(-1)^kd\sim X\\
M_g(d)^2L(1/2,f\otimes\chi_{d})\leq C_6}}\alpha_g^{1/2}\omega_g^{1/2}|c_g(|d|)|M_g(d)\geq C_7KX.\]

Let us now define the set
\[\mathcal V_g:=\left\{x\sim X:\,\sumflat_{\substack{x\leq(-1)^kd\leq x+H\\
M_g(d)^2L(1/2,f\otimes\chi_{d})\leq C_6}}\sqrt{\alpha_g}|c_g(|d|)|M_g(d)\geq\frac{C_8 H}{\sqrt k}\right\}.\]

From the work above it follows that, for large enough absolute constant $C_9>0$,
\begin{align*}
C_7 KX &\leq\sum_{k\in\Z}h_2\left(\frac{2k-1}K\right)\sum_{g\in B_{k+\frac12}^+}\sumflat_{\substack{(-1)^kd\sim X\\
M_g(d)^2L(1/2,f\otimes\chi_{d})\leq C_6}}\alpha_g^{1/2}\omega_g^{1/2}|c_g(|d|)|M_g(d)\\
&\leq\frac{C_9}H\sum_{k\in\Z}h_2\left(\frac{2k-1}K\right)\sum_{g\in B_{k+\frac12}^+}\sum_{x\sim X}\sumflat_{\substack{x\leq (-1)^kd\leq x+H\\
M_g(d)^2L(1/2,f\otimes\chi_{d})\leq C_6}}\alpha_g^{1/2}\omega_g^{1/2}|c_g(|d|)|M_g(d)\\
&=\frac{C_9}H\sum_{k\in\Z}h_2\left(\frac{2k-1}K\right)\sum_{g\in B_{k+\frac12}^+}\left(\sum_{x\in\mathcal V_g}+\sum_{x\not\in\mathcal V_g}\right)\sumflat_{\substack{x\leq (-1)^kd\leq x+H\\
M_g(d)^2L(1/2,f\otimes\chi_{d})\leq C_6}}\alpha_g^{1/2}\omega_g^{1/2}|c_g(|d|)|M_g(d)\\
&\leq\frac{C_9}H\sum_{k\in\Z}h_2\left(\frac{2k-1}K\right)\sum_{g\in B_{k+\frac12}^+}\left(\sqrt{C_6}\omega_g H\left|\mathcal V_g\right|+\omega_g^{1/2}\frac{C_8 XH}{\sqrt k}\right),
\end{align*}
where we have used the relation $\omega_g^{1/2}\alpha_g^{1/2}|c_g(|d|)|=\omega_g\sqrt{L(1/2,f\otimes\chi_d)}$ in the last estimate. Using an easy estimate $\sum_{g\in B_{k+\frac12}^+}\omega_g^{1/2}\leq C_{10}\sqrt k$ (which follows from the Cauchy--Schwarz inequality and the fact that $\sum_{g\in B_{k+\frac12}^+}\omega_g\sim 1$) we conclude that
\begin{align}\label{lower-bound}
\sum_{k\in\Z}h_2\left(\frac{2k-1}K\right)\sum_{g\in B_{k+\frac12}^+}\omega_g\left|\mathcal V_g\right|\geq C_{11} KX.
\end{align}
The harmonic weights $\omega_g$ can be removed affecting only the constant factor in the lower bound \cite{Balkanova-Frolenkov2021, Iwaniec-Sarnak2000}. Indeed, by the Cauchy--Schwarz inequality we have
\begin{align*}
\sum_{k\in\Z}h_2\left(\frac{2k-1}K\right)\sum_{g\in B_{k+\frac12}^+}\omega_g\left|\mathcal V_g\right|\leq \left(\sum_{k\in\Z}h_2\left(\frac{2k-1}K\right)\sum_{g\in B_{k+\frac12}^+}|\mathcal V_g|\right)^{1/2}\left(\sum_{k\in\Z}h_2\left(\frac{2k-1}K\right)\sum_{g\in B_{k+\frac12}^+}\omega_g^2|\mathcal V_g|\right)^{1/2}.
\end{align*}
Estimating trivially the second factor is
\begin{align*}
&\leq \sqrt{2X\sum_{k\in\Z}h_2\left(\frac{2k-1}K\right)\sum_{g\in B_{k+\frac12}^+}\omega_g^2}\\
&\leq\sqrt{4C_{12}X}
\end{align*}
as
\begin{align}\label{weights-squared}
\sum_{g\in B_{k+\frac12}^+}\omega_g^2\leq\frac{C_{12}}k.
\end{align}
This bound is well-known so we sketch the proof. Using the definition of $\omega_g$ it suffices to show that 
\[
\sum_{f\in B_k}\omega_f\frac1{L(1,\text{sym}^2f)}\ll 1. 
\]
To see this, one approximates $L(1,\text{sym}^2f)^{-1}$ by a short Dirichlet polynomial (see e.g. \cite[Proof of Lemma 8.9.]{Balkanova-Frolenkov2021}) and then applies the Petersson formula. This leads to a convergent sum, which finishes the proof of (\ref{weights-squared}).

Therefore, using (\ref{lower-bound}), we deduce that  
\begin{align}\label{V_g-on-av-2}
\sum_{k\in\Z}h_2\left(\frac{2k-1}K\right)\sum_{g\in B_{k+\frac12}^+}|\mathcal V_g|\geq C_{13} K^2X. 
\end{align}
Let us introduce the set
\[\mathcal U:=\left\{g\in\mathcal S_K:\,|\mathcal V_g|\geq C_{14}X\right\}.
\]
Now from (\ref{V_g-on-av-2}) we deduce that 
\begin{align*}
C_{13}XK^2&\leq\sum_{g\in\mathcal U}\left|\mathcal V_g\right|+\sum_{g\in \mathcal S_K\setminus\mathcal U}\left|\mathcal V_g\right|\\
&\leq |\mathcal U|X+C_{14}K^2 X
\end{align*}
from which we infer the lower bound 
\[|\mathcal U|\geq(C_{13}-C_{15})K^2
\]
by recalling that $C_{15}=C_{14}/4$ and $C_{14}$ is chosen so that $C_{15}<C_{13}$. 

Hence we have shown that for $\geq (C_{13}-C_{15})K^2$ of the forms $g\in \mathcal S_K$ we have
\begin{align*}
&\#\left\{x\sim X\,:\,\sumflat_{x\leq (-1)^kd\leq x+H}\sqrt{\alpha_g}|c_g(|d|)|M_g(d)\geq\frac{C_8 H}{\sqrt k}\right\}\\
&\geq\#\left\{x\sim X\,:\,\sumflat_{\substack{x\leq (-1)^kd\leq x+H\\
M_g(d)^2L(1/2,f\otimes\chi_d)\leq C_6}}\sqrt{\alpha_g}|c_g(|d|)|M_g(d)\geq\frac{C_8 H}{\sqrt k}\right\}\\
&\geq C_{14}X.
\end{align*}
In conclusion we have shown that for all but $\leq 8\sqrt{C_3}K^2/C_2C_1^2$ of the forms $g\in S_K$ it holds that 
\[ 
\left|\sumflat_{x\leq (-1)^kd\leq x+H}\sqrt{\alpha_g}c_g(|d|)M_g(d)\right|<\frac{C_1\sqrt H}{\sqrt k}
\]
for a positive proportion of $x\sim X$ with the exceptional set having size $\leq C_2X$. Moreover, for at least $(C_{13}-C_{15})K^2$ of the forms $g\in S_K$ it holds that 
\[ 
\sumflat_{x\leq (-1)^kd\leq x+H}\sqrt{\alpha_g}|c_g(|d|)|M_g(d)\geq\frac{C_8 H}{\sqrt k}
\]
for $\geq C_{14}X$ of the numbers $x\sim X$. By recalling that our choices of constants are so that $C_{13}-C_{15}>8\sqrt{C_3}/C_2C_1^2$ it follows that for $\geq (C_{13}-C_{15}-8\sqrt{C_3}/C_2C_1^2)K^2$ of the forms $g\in S_K$ the chain of inequalities
\begin{align}\label{ineq-chain}
\left|\sumflat_{x\leq (-1)^kd\leq x+H}\sqrt{\alpha_g}c_g(|d|)M_g(d)\right|<\frac{C_1\sqrt H}{\sqrt k}<\frac{C_8 H}{\sqrt k}\leq \sumflat_{x\leq (-1)^kd\leq x+H}\sqrt{\alpha_g}|c_g(|d|)|M_g(d) 
\end{align}
holds for $\geq (C_{14}-C_2)X$ of the numbers $x\sim X$ if we choose $H>(C_1/C_8)^2$.

Choose $H$ to be large enough ($>(C_1/C_8)^2$), but fixed. With this choice it follows easily from (\ref{ineq-chain}) that for $\geq (C_{13}-C_{15}-8\sqrt{C_3}/C_2C_1^2)K^2$ of the forms $g\in\mathcal S_K$ we have that
\[ 
\left|\sumflat_{x\leq (-1)^kd\leq x+H}\sqrt{\alpha_g}|c_g(|d|)|M_g(d)\,\pm\sumflat_{x\leq (-1)^kd\leq x+H}\sqrt{\alpha_g}c_g(|d|)M_g(d)\right|\geq \frac{(C_8H-C_1\sqrt H)}{\sqrt k}>0
\]
holds for $\geq (C_{14}-C_2)X$ of the numbers $x\sim X$. We note that the contribution coming from the summands with $\sqrt{\alpha_g}|c_g(|d|)\leq k^{-\delta}$ is trivially bounded by, say,
\[
\leq \sum_{\substack{x\leq (-1)^k d\leq x+H\\
\sqrt{\alpha_g} |c_g(|d|)|\leq k^{-\delta}}}\sqrt{\alpha_g}|c_g(|d|)|M_g(d)\ll K^{-\delta} \sum_{\substack{x\leq (-1)^k d\leq x+H\\
\sqrt{\alpha_g} |c_g(|d|)|\leq k^{-\delta}}}M_g(d)\ll  HK^{-\delta+\delta_0}(\log K)^{1/4}.
\]
Choosing $\delta=1/2+\eta/2$ for concreteness so that $\delta>1/2$ and $\eta_2$ so that $\delta_0<\eta/2$, we conclude that for the same proportion of $g\in \mathcal S_K$ and $x\sim X$ we have
\[
\left|\sumflat_{\substack{x\leq (-1)^kd\leq x+H\\
\sqrt{\alpha_g}|c_g(|d|)|> k^{-\delta}}}\sqrt{\alpha_g}|c_g(|d|)|M_g(d)\,\pm\sumflat_{\substack{x\leq (-1)^kd\leq x+H\\
\sqrt{\alpha_g}|c_g(|d|)|> k^{-\delta}}}\sqrt{\alpha_g}c_g(|d|)M_g(d)\right|\geq \frac {(C_8H-C_1\sqrt H)}{\sqrt k}>0
\]
for sufficiently large $k$. Now observe that this implies 
\begin{align*}
& 2\sumflat_{\substack{x\leq (-1)^kd\leq x+H\\
\sqrt{\alpha_g}c_g(|d|)> k^{-\delta}}}\sqrt{\alpha_g}|c_g(|d|)|M_g(d)\\
&=\left|\sumflat_{\substack{x\leq (-1)^kd\leq x+H\\
\sqrt{\alpha_g}|c_g(|d|)|> k^{-\delta}}}\sqrt{\alpha_g}|c_g(|d|)|M_g(d)\,+\sumflat_{\substack{x\leq (-1)^kd\leq x+H\\
\sqrt{\alpha_g}|c_g(|d|)|> k^{-\delta}}}\sqrt{\alpha_g}c_g(|d|)M_g(d)\right|\\
&\geq \frac{(C_8 H-C_1\sqrt H)}{\sqrt k}>0.
\end{align*}
Similarly,
\begin{align*}
& 2\sumflat_{\substack{x\leq (-1)^kd\leq x+H\\
\sqrt{\alpha_g}c_g(|d|)< -k^{-\delta}}}\sqrt{\alpha_g}|c_g(|d|)|M_g(d)\\
&=\left|\sumflat_{\substack{x\leq (-1)^kd\leq x+H\\
\sqrt{\alpha_g}|c_g(|d|)|> k^{-\delta}}}\sqrt{\alpha_g}|c_g(|d|)|M_g(d)\,-\sumflat_{\substack{x\leq (-1)^kd\leq x+H\\
\sqrt{\alpha_g}|c_g(|d|)|> k^{-\delta}}}\sqrt{\alpha_g}c_g(|d|)M_g(d)\right|\\
&\geq \frac{(C_8H-C_1\sqrt H)}{\sqrt k}>0.
\end{align*}
Thus we have shown that for $\geq(C_{13}-C_{15}-8\sqrt{C_3}/C_2C_1^2)K^2$ of the forms $g\in\mathcal S_K$ the short interval $[x,x+H]$ contains numbers $(-1)^k d_\pm$, with both $d_\pm$ odd fundamental discriminants, for which $\sqrt{\alpha_g}c_g(|d_+|)>k^{-\delta}$ and $\sqrt{\alpha_g}c_g(|d_-|)<-k^{-\delta}$, for $\geq (C_{14}-C_2)X$ of the numbers $x\sim X$. Thus we deduce a sign change of $c_g(|d|)$ over a positive proportion of intervals of constant size $H>(C_1/C_8)^2$ for a positive proportion of forms $g\in S_K$. This leads to 
\[ 
\geq \frac{(C_{14}-C_2)X}{H}\asymp\frac{K}{Y}
\]
real zeroes on both of the line segments $\delta_1$ and $\delta_2$, concluding the proof. \qed

\bibliography{Half_integral_real_zeros_II}
\bibliographystyle{plain}

\end{document}